\documentclass[a4paper,11 pt,reqno]{amsart}
\usepackage{a4,amsbsy,amscd,amssymb,amstext,amsmath,mathtools,mathdots,latexsym,oldgerm,enumerate,verbatim,graphicx,xy,ytableau}

\makeatletter                                       
\def\l@subsection{\@tocline{2}{0pt}{2.5pc}{2.5pc}{}}
\makeatother                                        

\makeatletter                                                  
\def\chapter{\clearpage\thispagestyle{plain}\global\@topnum\z@ 
\@afterindenttrue \secdef\@chapter\@schapter}
\makeatother

\newtheorem{thmgl} {Theorem}    
\newtheorem{propgl}{Proposition}
\newtheorem{lemgl} {Lemma}
\newtheorem{corgl} {Corollary}

\newtheorem{lemnn}{Lemma}

\theoremstyle{definition}

\newtheorem{remgl} {Remark}
\newtheorem{remsgl} [remgl]{Remarks}

\newcommand{\mf}{\mathfrak}
\newcommand{\mc}{\mathcal}
\newcommand{\mb}{\mathbb}

\newcommand{\nts}{\negthinspace}     
\newcommand{\Nts}{\nts\nts}

\newcommand{\ov}{\overline}
\newcommand{\un}{\underline}

\newcommand{\sm}{\setminus}         

\newcommand{\ot}{\otimes}           
\newcommand{\la}{\langle}
\newcommand{\ra}{\rangle}

\newcommand{\Hom}{{\rm Hom}}        
\newcommand{\Mor}{{\rm Mor}}
\newcommand{\End}{{\rm End}}

\newcommand{\Mat}{{\rm Mat}}

\newcommand{\Ext}{{\rm Ext}}

\newcommand{\Sym}{{\rm Sym}} 

\renewcommand{\Im}{{\rm Im}}

\newcommand{\rk}{{\rm rk}}

\newcommand{\tr}{{\rm tr}}
\newcommand{\id}{{\rm id}}

\newcommand{\g}{\mf{g}}

\let\ttie\t
\newcommand{\tie}[1]{{\let\t\ttie \ttie#1}}
\renewcommand{\t}{\mf{t}}  

\newcommand{\gl}{\mf{gl}}

\newcommand{\Dist}{{\rm Dist}}
\newcommand{\GL}{{\rm GL}}

\newcommand{\SL}{{\rm SL}}

\newcommand*\bmat[1]{\begin{bsmallmatrix}#1\end{bsmallmatrix}}

\makeatletter
\def\vcdots{\vbox{\baselineskip4\p@ \lineskiplimit\z@
\kern3\p@\hbox{.}\hbox{.}\hbox{.}\Nts\nts\kern3\p@}}
\makeatother

\xyoption{all}

\begin{document}

\title{Highest weight vectors 
and transmutation}

\begin{abstract}
Let $G=\GL_n$ be the general linear group over an algebraically closed field $k$, let $\g=\gl_n$ be its Lie algebra and
let $U$ be the subgroup of $G$ which consists of the upper uni-triangular matrices. Let $k[\g]$ be the algebra of polynomial
functions on $\g$ and let $k[\g]^G$ be the algebra of invariants under the conjugation action of $G$.
We consider the problem of giving finite homogeneous spanning sets for the $k[\g]^G$-modules of highest weight vectors for
the conjugation action on $k[\g]$. We prove a general result in arbitrary characteristic which reduces the problem to
giving spanning sets for the vector spaces of highest weight vectors for the action of
$\GL_r\times\GL_s$ on tuples of $r\times s$ matrices. This requires the technique called ``transmutation" by R.~Brylinsky which
is based on an instance of Howe duality.
In characteristic zero,
we give for all dominant weights $\chi\in\mathbb Z^n$ finite homogeneous spanning sets for the $k[\g]^G$-modules $k[\g]_\chi^U$
of highest weight vectors. This result was already stated by J.~F.~Donin, but he only gave proofs
for his related results on skew representations for the symmetric group.
We do the same for tuples of $n\times n$-matrices under the diagonal conjugation action.
\end{abstract}

\author[R.\ H.\ Tange]{Rudolf Tange}

\keywords{}
\thanks{2010 {\it Mathematics Subject Classification}. 13A50, 16W22, 20G05.}

\maketitle

\section*{Introduction}\label{s.intro}
Let $k$ be  an algebraically closed field and let $\GL_n$ be the group of invertible $n\times n$ matrices with entries in $k$ and let $T_n$ and $U_n$ be the subgroups of diagonal matrices and of upper uni-triangular matrices. The group $\GL_n$ acts on the $k$-vector space $\Mat_n$ of $n\times n$ matrices with entries in $k$ via $S\cdot A=SAS^{-1}$ and therefore on its coordinate ring $k[\Mat_n]$ via $(S\cdot f)(A)=f(S^{-1}AS)$. We identify the character group of $T_n$ with $\mb Z^n$: if $\chi\in\mb Z^n$, then $D\mapsto \prod_{i=1}^nD_{ii}^{\chi_i}$ is the corresponding character of $T_n$. We will call the characters of $T_n$ {\it weights} of $T_n$ or $\GL_n$ and the weights $\chi$ of $T_n$ for which the corresponding weight space $M_\chi$ of a given $T_n$ module $M$ is nonzero will be called {\it weights} of $M$. We say that $\chi\in\mb Z^n$ is {\it dominant} if it is weakly decreasing.

The study of the polynomial ring $k[\g]$ as a $G$-module for a reductive group $G$ with Lie algebra $\g$ under the adjoint action was initiated in Kostant's landmark paper \cite{Kos}.
We will be interested in finding finite homogeneous spanning sets for the $k[\Mat_n]^{\GL_n}$-modules $k[\Mat_n]_\chi^{U_n}$ of highest weight vectors. As is well-known, such a module is nonzero if and only if $\chi$ is dominant and has coordinate sum zero. A weight $\chi\in\mb Z^n$ with this property can uniquely be written as $\chi=[\lambda,\mu]=[\lambda,\mu]_n:=(\lambda_1,\lambda_2,\ldots,0,\ldots,0,\ldots,-\mu_2,-\mu_1)$ were $\lambda$ and $\mu$ are partitions with $|\lambda|=|\mu|$ and $l(\lambda)+l(\mu)\le n$. Here $l(\lambda)$ denotes the length of a partition $\lambda$ and $|\lambda|$ denotes its coordinate sum. As usual partitions are extended with zeros if necessary.

The nilpotent cone $\mc N_n=\{A\in\Mat_n\,|\,A^n=0\}$ is a $\GL_n$-stable closed subvariety of $\Mat_n$. Using the graded Nakayama Lemma it is easy to see that it suffices to find finite homogeneous spanning sets for the vector spaces of highest weight vectors $k[\mc N_n]_\chi^{U_n}$ in the coordinate ring of $\mc N_n$. For background on the conjugation action of $\GL_n$ on $k[\Mat_n]$ and $k[\mc N_n]$, e.g. graded character formulas, we refer to the introduction of \cite{T2} and the references in there.

In \cite{Br} a process called {\it transmutation} is applied to understand the conjugation action of $\GL_n$ on the nilpotent cone. We briefly explain the idea and for simplicity we assume that $k$ has characteristic $0$. Let $G,H$ be reductive groups and let $Y$ be an affine $G\times H$-variety such that $k[Y]=\bigoplus_{i\in I}L_i^*\ot M_i$ where the $L_i$
are mutually nonisomorphic $G$-modules and the $M_i$ are mutually nonisomorphic $H$-modules.
Then $Y$ can be used as a ``catalyst" for transmutation as follows. If $V$ is an affine $G$-variety, then $W=Y\times^GV:=(Y\times V)//G$ is an affine $H$-variety, the $H$-irreducibles that show up in $k[W]$ are the $M_i$, and the multiplicity of $M_i$ in $k[W]$ is the same as that of $L_i$ in $k[V]$. The goal is to find for a given $V$ a suitable $H$ and $Y$ for which the resulting $W$ is much simpler than $V$, but still contains enough interesting information coming from $V$. In \cite{Br} R.~Brylinsky applied this technique to the closed $\GL_n$-stable subvariety $V=\mc N_{n,m}=\{A\in\mc N_n\,|\,A^{m+1}=0\}$ of $\Mat_n$ and $G=\GL_n$. She showed that in this case for $H=\GL_r\times\GL_s$ and a suitable catalyst $Y$ the transmuted variety $W$ is a certain closed subvariety of $\Mat_{rs}^m$ which is all of $\Mat_{rs}^m$ if $n$ is sufficiently big relative to $m,r$ and $s$. Here $\GL_r\times\GL_s$ acts on $\Mat_{rs}^m$ via $((R,S)\cdot\un A)_i=RA_iS^{-1}$, $\un A=(A_1,\ldots,A_m)\in\Mat_{rs}^m$, and on the coordinate ring $k[\Mat_{rs}^m]$ via $((R,S)\cdot f)(\un A)=f((R^{-1},S^{-1})\cdot\un A)$.
The correspondence between the irreducibles for the two groups is in terms of the labels given by $\chi=[\lambda,\mu]\leftrightarrow(-\mu^{\rm rev},\lambda)$, where $\mu^{\rm rev}$ is the reversed $r$-tuple of $\mu$.

In this paper we give finite homogeneous spanning sets for the vector spaces $k[\mc N_n]_\chi^{U_n}$ in characteristic $0$ (Corollary~2 to Theorem~\ref{thm.highest_weight_vecs}) using ``transmutation" (Theorem~\ref{thm.surjective_pullback}) and J.~Donin's results on skew representations for the symmetric group, see Section~\ref{ss.Sym}. For this it is necessary that we make Brylinsky's work explicit in terms of highest weight vectors. It turns out that the method of ``transmutation" works in our case in any characteristic and for certain special weights we can give bases for the highest weight vectors in the coordinate ring of the transmuted variety which then give spanning sets for the highest weight vectors in the coordinate ring of $\mc N_n$.

The paper is organised as follows. In Section~\ref{s.prelim} we introduce some notation, e.g. for diagrams and tableaux, and we state some well-known results from the literature on the invariant algebra $k[\Mat_n]^{\GL_n}$, reduction to the nilpotent cone and good filtrations that we will need.
In Section~\ref{s.charp} we show in Theorem~\ref{thm.surjective_pullback} that the technique of transmutation works in our case in any characteristic. Our main tool here is Donkin's results on good pairs of varieties \cite{Don1}. We can apply Theorem~\ref{thm.surjective_pullback} in arbitrary characteristic for weights $\chi$ with $\chi_n\ge-1$ or $\chi_1\le 1$. For the corresponding $\GL_r\times\GL_s$-weights we give in Theorem~\ref{thm.basis_special_weights} bases for the spaces of highest weight vectors in the coordinate ring of the ``transmuted space" $\Mat_{rs}^m$.

In Section~\ref{s.char0} we always assume that our field $k$ has characteristic $0$. In Section~\ref{ss.Sym} we first develop the necessary results on skew representations of the symmetric group. What we need is explicit polytabloid bases for the ``coinvariants" for a Young subgroup in a tensor product of Specht modules, see Proposition~\ref{prop.coinvariants}. In Section~\ref{ss.highest_weight_vecs} we give in Theorem~\ref{thm.highest_weight_vecs} bases for the spaces of highest weight vectors in the coordinate ring of the ``transmuted space" $\Mat_{rs}^m$. Combined with Theorem~\ref{thm.surjective_pullback} this gives finite homogeneous spanning sets for the vector spaces $k[\mc N_n]_\chi^{U_n}$ in characteristic $0$, see Corollary~2. This can then further be combined with Lemma~\ref{lem.reduction_to_nilpotent_cone} to obtain finite homogeneous spanning sets for the $k[\Mat_n]^{\GL_n}$-modules $k[\Mat_n]_\chi^{U_n}$, see Corollary~3. In Section~\ref{ss.several_matrices} we briefly describe a generalisation to several matrices and how to obtain spanning sets for the $k[\Mat_n^l]^{\GL_n}$-modules $k[\Mat_n^l]_\chi^{U_n}$.

I now explain the relation of Section~\ref{ss.Sym} and Corollary~3 to Theorem~\ref{thm.highest_weight_vecs} with Donin's work \cite{Donin1, Donin2}. 
Donin gave proofs in \cite{Donin1} for his results on skew representations for the symmetric group, but these proofs are often incomplete and \cite{Donin1} was never published.
The paper \cite{Donin2} contains no proofs. Therefore I have given an account with complete proofs in Section~\ref{ss.Sym}.
Especially in the proof of Theorem~\ref{thm.homspace_basis} I follow Donin's approach closely. In all cases a reference to the corresponding result from Donin is given if there is one.
Furthermore, some inaccuracies have been corrected, see e.g. Remark~\ref{rem.coinvariants}.
Corollary~3 to Theorem~\ref{thm.highest_weight_vecs} which describes spanning sets for the $k[\Mat_n]^{\GL_n}$-modules $k[\Mat_n]_\chi^{U_n}$ is also stated by Donin in \cite{Donin1, Donin2}\footnote{Actually Donin claimed that they are bases, but this is incorrect, see Remark~\ref{rems.highest_weight_vecs}.2.},
but the proof sketch given in \cite[p31,32]{Donin1} is unconvincing and no logical link is made with his results on the symmetric group.
In our approach we derive this result using transmutation (Theorem~\ref{thm.surjective_pullback}) from a result (Theorem~\ref{thm.highest_weight_vecs}) on the highest weight vectors in the coordinate ring of a completely different variety with group action: $\Mat_{rs}^m$ under the action of $\GL_r\times\GL_s$.
The latter result is then proved using Donin's results on skew representations for the symmetric group.

\section{Preliminaries}\label{s.prelim}
Throughout this paper $k$ is an algebraically closed field. All our varieties are affine. The groups $\GL_n,T_n,U_n$ and the actions of $\GL_n$ on $\Mat_n$ and $\mc N_{n,m}$ and of $\GL_r\times\GL_s$ on $\Mat_{rs}^m$ are as in the introduction.

For $G$ a reductive group and $\chi$ a dominant weight relative to a Borel subgroup $B=TU$ we denote the standard or Weyl module corresponding to $\chi$ by $\Delta_G(\chi)$ and the costandard or induced module corresponding to $\chi$ by $\nabla_G(\chi)$. We have $\Delta_G(\chi)\cong\nabla_G(-w_0(\chi))^*$, where $w_0$ is the longest element in the Weyl  group. The module $\nabla_G(\chi)$ has simple socle and the module $\Delta_G(\chi)$ has simple top, both isomorphic to the irreducible $L_G(\chi)$ of highest weight $\chi$. In characteristic $0$ we have $\Delta_G(\chi)\cong\nabla_G(\chi)\cong L_G(\chi)$. The main property of these modules that we will use is that for all dominant $\chi_1$ and $\chi_2$, $\Ext_G^1(\Delta_G(\chi_1),\nabla_G(\chi_2))=0$ and $\Hom_G(\Delta_G(\chi_1),\nabla_G(\chi_2))=k$ if $\chi_1=\chi_2$ and $\{0\}$ otherwise. See \cite[II.4.13]{Jan}.
\subsection{The graded Nakayama Lemma}
As is well-known the algebra $k[\Mat_n]^{\GL_n}$ is generated by the algebraically independent functions $s_1,\ldots,s_n$ given by $s_i(A)=\tr(\wedge^iA)$, where $\wedge^iA$ denotes the $i$-th exterior power of $A$. Furthermore, the $s_i$ generate the vanishing ideal of $\mc N_n$. If $m$ is the dimension of the zero weight space of $\nabla_{\GL_n}(\chi)$, then $k[\mc N_n]^{U_n}_\chi$ has dimension $m$ and $k[\Mat_n]^{U_n}_\chi$ is a free $k[\Mat_n]^{\GL_n}$-module of rank $m$. The following lemma is an application of the graded Nakayama Lemma.
\begin{lemgl}\label{lem.reduction_to_nilpotent_cone}
Let $f_1,\ldots,f_l\in k[\Mat_n]^{U_n}_\chi$ be homogeneous. If the restrictions $f_1|_{\mc N_n},\ldots,f_l|_{\mc N_n}$ span $k[\mc N_n]^{U_n}_\chi$, then $f_1,\ldots,f_l$ span $k[\Mat_n]^{U_n}_\chi$ as a $k[\Mat_n]^{\GL_n}$-module.
The same holds with ``span" replaced by ``form a basis of".
\end{lemgl}
\noindent We refer to \cite[Lem.~2, Prop.~1]{T1} for references and explanation.
\subsection{Good filtrations}
A $G$-module $M$ is said to have a {\it good filtration} if it has a (possibly finite) $G$-module filtration $0=M_0\subseteq M_1\subseteq M_2\subseteq\cdots$, $\bigcup_{i\ge0}M_i=M$, such that each quotient $M_i/M_{i-1}$ is isomorphic to some induced module $\nabla_G(\chi)$. If $M$ has a good filtration, the number of quotients isomorphic to $\nabla_G(\chi)$ is independent of the good filtration and equals $\dim M^U_\chi$. If $k$ has characteristic $0$, then every $G$-module has a good filtration. For more details we refer to \cite[II.4.16,17]{Jan}. For example, a direct summand of a module with a good filtration has a good filtration.
\subsection{Graded characters}
If $M=\bigoplus_{i\ge 0} M_i$ is a graded vector space with $\dim M_i<\infty$ for all $i$, then the {\it graded dimension} of $M$ is the formal power series $\sum_{i\ge 0}\dim M_iz^i$. Here one can use for $z$ any other grading variable. Similarly, if $G$ is a general linear group, $M=\bigoplus_{i\ge 0} M_i$ a graded $G$-module with a good filtration, and $\nabla_G(\chi)$ has finite good filtration multiplicity in $M$, then the {\it graded good filtration multiplicity} of $\nabla_G(\chi)$ in $M$ is the formal power series $\sum_{i\ge 0} (M_i:\nabla_G(\chi))z^i$, where $(M_i:\nabla_G(\chi))$ is the good filtration multiplicity of $\nabla_G(\chi)$ in $M_i$. Note that by the above the graded good filtration multiplicity of $\nabla_G(\chi)$ in $M$ is the graded dimension of $M^U_\chi$. We say that one graded dimension or multiplicity is $\le$ another if this is true coefficient-wise.
\subsection{Good pairs}
Recall from \cite{Don1} that an affine variety $V$ on which a reductive group $G$ acts is called {\it good} if $k[V]$ has a good filtration.
Furthermore, if $A$ is a closed $G$-stable subvariety of $V$, then $(V,A)$ is called a {\it good pair of $G$-varieties} if the
vanishing ideal of $A$ in $k[V]$ has a good filtration. In this case $A$ is itself a good $G$-variety.
If $(V,A)$ is a good pair of $G$-varieties, then the restriction map $k[V]^U_\chi\to k[A]^U_\chi$ is surjective by \cite[II.4.13]{Jan}.
\subsection{Skew Young diagrams and tableaux}
For $\lambda$ a partition of $n$ we denote the nilpotent orbit which consists of the matrices whose Jordan normal form has block sizes $\lambda_1,\cdots,\lambda_{l(\lambda)}$, by $\mc O_\lambda$. For $\lambda,\mu$ partitions of $n$, we say that $\lambda\ge\mu$ if $\sum_{j=1}^i\lambda_j\ge\sum_{j=1}^i\mu_j$ for $i=1,\ldots,n-1$. This order is called the {\it dominance order}. In \cite[Prop~1.6]{Ger} it was proved that $\ov{\mc O}_\lambda\supseteq\mc O_\mu$ if and only if $\lambda\ge\mu$. 
Here $\ov{\mc O}_\lambda$ denotes the closure of the orbit $\mc O_\lambda$.
Since $\mc N_{n,m-1}$ is the union of the $\mc O_\lambda$ with $\lambda_1\le m$, it follows  easily that $\mc N_{n,m-1}=\ov{\mc O}_{m^qr}$, where $q$ and $r$ are quotient and remainder under division of $n$ by $m$.

We will denote the transpose of a partition $\lambda$ by $\lambda'$ and we will identify each partition $\lambda$ with the corresponding Young diagram $\{(i,j)\,|\,1\le i\le l(\lambda),1\le j\le\lambda_i\}$.
The $(i,j)\in\lambda$ are called the {\it boxes} or {\it cells} of $\lambda$.
More generally, if $\lambda,\mu$ are partitions with $\lambda\supseteq\mu$, then we denote the diagram $\lambda$ with the
boxes of $\mu$ removed by $\lambda/\mu$ and call it the {\it skew Young diagram} associated to the pair $(\lambda,\mu)$.
Of course the skew diagram $\lambda/\mu$ does not determine $\lambda$ and $\mu$. We denote the number of boxes in
a skew diagram $E$ by $|E|$. We define $\Delta_t$ to be the diagram
$$
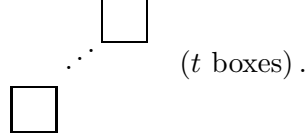
\begin{ytableau}
\none&\none&\ \\
\none&\none[\iddots]&\none[\text {\hspace{3cm} ($t$ boxes)\,.}]\\
\ &\none&\none
\end{ytableau}$$
\smallskip

Let $E$ be a skew diagram with $t$ boxes. A {\it skew tableau} of shape $E$ is a mapping $T:E\to \mb N=\{1,2,\ldots\}$.
A skew tableau of shape $E$ is called {\it row-ordered} if its entries are weakly increasing along rows,
{\it strictly row-ordered} if its entries are strictly increasing along rows,
and it is called {\it ordered} if its entries are weakly increasing along rows and down columns.
The notions column-ordered and strictly column-ordered are defined in a completely analogous way.
A skew tableau of shape $E$ is called {\it semi-standard} if its entries are weakly increasing along
the rows and strictly increasing down the columns, and it is called {\it row semi-standard} if its entries
are strictly increasing along the rows and weakly increasing down the columns.
It is called a {\it $t$-tableau} if its entries are the numbers $1,\ldots,t$ (so the entries
must be distinct) and it is called {\it standard} if it is a $t$-tableau and its entries are (strictly) increasing
along rows and down columns. We will associate to $E$ two special skew tableaux $T_E$ and $S_E$ as follows.
We define $T_E$ by filling in the numbers $1,\ldots,t$ row by row from left to right and top to bottom
and we define $S_E$ by filling the boxes in the $i$-th row with $i$'s. So $T_E$ is standard and $S_E$ is semi-standard.
Two tableaux $S$ and $T$ of shape $E$ are called {\it row equivalent} if, for each $i$, the $i$-th row of $F$ is
a permutation of the $i$-th row of $T$. The notion of column equivalence is defined in a completely analogous way.
Finally, if $m$ is the biggest integer occurring in a tableau $T$, or $0$ if $T$ is empty, then the {\it weight} of $T$
is the $m$-tuple whose $i$-th component is the number of occurrences of $i$ in $T$. Sometimes we will also consider
the weight of $T$ as an $m'$-tuple for some $m'\ge m$ by extending it with zeros.

\section{Transmutation and semi-invariants in arbitrary characteristic}\label{s.charp}
Let $r,s$ be integers $\ge0$ with $r+s\le n$. We denote the variety of pairs $(A,B)\in\Mat_{rn}\times\Mat_{ns}$ with
$AB=0$ by $Y_{r,s,n}$ and for $m$ an integer $\ge 2$ we define the maps $\varphi_{r,s,n,m}$ and $\ov\varphi_{r,s,n,m}$ by
\begin{align*}
\varphi_{r,s,n,m}:(A,B,X)\mapsto&(AB,AXB,\ldots,AX^mB)\\
&:\Mat_{rn}\times\Mat_{ns}\times\Mat_n\to\Mat_{rs}\times\Mat_{rs}^m\\
\ov\varphi_{r,s,n,m}:(A,B,X)\mapsto&(s_1(X),\ldots,s_n(X),\varphi_{r,s,n,m}(A,B,X))\\
&:\Mat_{rn}\times\Mat_{ns}\times\Mat_n\to k^n\times\Mat_{rs}\times\Mat_{rs}^m\,.
\end{align*}
We will denote several of the restrictions of these maps by the same symbol. The group $\GL_{r,s,n}:=\GL_r\times\GL_s\times\GL_n$
acts on $\Mat_{rn}\times\Mat_{ns}$ via $(S,T,U)\cdot(A,B)=(SAU^{-1},UBT^{-1})$ and on $\Mat_{rn}\times\Mat_{ns}\times\Mat_n$ via
$(S,T,U)\cdot(A,B,X)=(SAU^{-1},UBT^{-1},UXU^{-1})$. Note that $Y_{r,s,n}$ is a $\GL_{r,s,n}$-stable closed subvariety of $\Mat_{rn}\times\Mat_{ns}$.
Note also that $\varphi_{r,s,n,m}$ and $\ov\varphi_{r,s,n,m}$ are equivariant for the action of $\GL_{r,s,n}$ if we let $\GL_n$ act trivially on
$ k^n\times\Mat_{rs}\times\Mat_{rs}^m$ and $\GL_r\times\GL_s$ trivially on $k^n$ and via its obvious diagonal action on $\Mat_{rs}\times\Mat_{rs}^m$.

We consider $\Mat_{rs}\times\Mat_{rs}^m$ as a closed subvariety of $k^n\times\Mat_{rs}\times\Mat_{rs}^m$
by taking the first $n$ scalar components zero and we consider $\Mat_{rs}^m$ as a closed subvariety of
$\Mat_{rs}\times\Mat_{rs}^m$ by taking the first matrix component the zero matrix.
So $\varphi_{r,s,n,m}=\ov\varphi_{r,s,n,m}$ on $\Mat_{rn}\times\Mat_{ns}\times\mc N_n$ and $\varphi_{r,s,n,m}(Y_{r,s,n}\times\mc N_n)\subseteq\Mat_{rs}^m$.
If $l\ge m$, then we consider $\Mat_{rs}^m$ as a closed subvariety of $\Mat_{rs}^l$ by extending an
$m$-tuple of $r\times s$ matrices with zero matrices to an $l$-tuple of $r\times s$ matrices.
So $\varphi_{r,s,n,l}=\varphi_{r,s,n,m}$ on $\Mat_{rn}\times\Mat_{ns}\times\mc N_{n,m}$ if $l\ge\min(m,n-1)$.
When $r$ and $s$ are fixed we denote the image $\varphi_{r,s,n,m}(Y_{r,s,n}\times\mc N_{n,m})\subseteq\Mat_{rs}^m$ by $W_{n,m}$.

We will use the embedding of $\Mat_n$ in $Y_{r,s,n}\times\Mat_n$ which is given by
$$X\mapsto(E_r,F_s,X)\,,$$
where $E_r=\begin{bmatrix}0\Nts&I_r\end{bmatrix}\in\Mat_{rn}$, $F_s=\begin{bmatrix}I_s\\0\end{bmatrix}\in\Mat_{ns}$.
Then $\varphi_{r,s,n,m}$ can be restricted to $\Mat_n$ and $\varphi_{r,s,n,m}(X)$ consists of the lower left $r\times s$ corners of the first $m$ powers of $X$.

Any point of $Y_{r,s,n}$ is contained in an irreducible curve which also contains a point $(A,B)\in Y_{r,s,n}$
with $A$ and $B$ of maximal rank $r$ and $s$ (see e.g. \cite[p38]{Br}) and if $(A,B)$ is such a point, then
it is easy to see that $g\cdot(A,B)=(E_r,F_s)$ for some $g\in\GL_n$. It follows that $Y_{r,s,n}$ is irreducible and that
$\varphi_{r,s,n,m}(\mc N_{n,m})$ is dense in $W_{n,m}$.

We will use the $\GL_{r,s,n}$-variety $Y_{r,s,n}$ as the catalyst for the transmutation
from $\GL_n$-varieties to $\GL_r\times\GL_s$-varieties. We will mainly be interested in applying this transmutation to the varieties $\mc N_{n,m}$.
Assertion (ii) of the next proposition, which is an analogue in arbitrary characteristic of \cite[Cor.~4.3]{Br}, says in particular that $W_{n,m}$
is the transmuted variety of $\mc N_{n,m}$. 

\begin{propgl}\label{prop.quotient}\ 
\begin{enumerate}[{\rm (i)}]
\item If $m\ge n-1$, then $\ov\varphi_{r,s,n,m}:\Mat_{rn}\times\Mat_{ns}\times\Mat_n\to k^n\times\Mat_{rs}\times\Mat_{rs}^m$ is a $\GL_n$-quotient
morphism onto its image.
\item If $r+s\le n$ and $\nu$ is a partition of $n$ with $\nu_1\le m+1$, then $Y_{r,s,n}\times\ov{\mc O}_\nu$ is a good $\GL_{r,s,n}$-variety
and $\varphi_{r,s,n,m}:Y_{r,s,n}\times\ov{\mc O}_\nu\to\Mat_{rs}^m$ is a $\GL_n$-quotient morphism onto its image.
\end{enumerate}
\end{propgl}
\begin{proof}
(i).\ If we apply \cite[Prop]{Don2} to the quiver with two nodes $x_1$ and $x_2$ of dimensions $1$ and $n$
with $s$ arrows from $x_1$ to $x_2$, $1$ loop at $x_2$ and $r$ arrows from $x_2$ to $x_1$, then we obtain that
the algebra	of $\GL_n$-invariants of $s$ vectors, $r$ covectors and $1$ matrix is generated by $s_1(X),\ldots,s_n(X)$ and the scalar products
$\la f,X^iv\ra$, where $f$ is one of the covectors, $v$ is one of the vectors, $X$ is the matrix and $i$ is $\ge0$.
Of course we may assume that $i<n$ by the Cayley-Hamilton Theorem. So we obtain the assertion.\\
(ii).\ As is well-known $\Mat_{rn}$ is a good $\GL_r\times\GL_n$-variety and therefore it is also a good $\GL_{r,s,n}$-variety
if we let $\GL_s$ act trivially. Similarly, $\Mat_{ns}$ is also a good $\GL_{r,s,n}$-variety and $\Mat_n$ is a good
$\GL_{r,s,n}$-variety if we let $\GL_r\times\GL_s$ act trivially. So, by the Donkin-Mathieu result on tensor products \cite[Prop.~II.4.21]{Jan}, 
$\Mat_{rn}\times\Mat_{ns}\times\Mat_n$ is a good $\GL_{r,s,n}$-variety.
Since $r+s\le n$, $Y_{r,s,n}$ is a good complete intersection in $\Mat_{rn}\times\Mat_{ns}$ by similar, but easier, arguments to those
in the proof of \cite[Thm.~2.1(c)]{Don1}. So $(\Mat_{rn}\times\Mat_{ns},Y_{r,s,n})$
is a good pair of $\GL_{r,s,n}$-varieties by \cite[Prop.~1.3b(i)]{Don1}. 
Furthermore, $(\Mat_n,\ov{\mc O}_\nu)$ is a good pair of $\GL_n$-varieties by \cite[Thm~2.2a(ii)]{Don1} and therefore also
a good pair of $\GL_{r,s,n}$-varieties if we let $\GL_r\times\GL_s$ act trivially. 
So $(\Mat_{rn}\times\Mat_{ns}\times\Mat_n,Y_{r,s,n}\times\ov{\mc O}_\nu)$ is a good pair of $\GL_{r,s,n}$-varieties by \cite[Prop~1.3e(i)]{Don1}.
This implies the first assertion and if we combine it with (i) and \cite[Prop~1.4a]{Don1} we obtain the second assertion.
\end{proof}

\begin{propgl}\label{prop.good_pair}\ 
Assume $r+s\le n$ and let $\nu$ be a partition of $n$ with $\nu_1\le m+1$.
Then $\big(\Mat_{rs}^m,\varphi_{r,s,n,m}(Y_{r,s,n}\times\ov{\mc O}_\nu)\big)$ is a good pair of $\GL_r\times\GL_s$-varieties.
\end{propgl}
\begin{proof}
Choose $N\ge (m+1)\max(r,s)$. By the argument in the proof of \cite[Thm~5.1]{Br} we have $\varphi_{r,s,N,m}(\mc N_{N,m})=\Mat_{rs}^m$
and therefore we certainly have $\varphi_{r,s,N,m}(Y_{r,s,N}\times\mc N_{N,m})=\Mat_{rs}^m$.
In the proof of Proposition~\ref{prop.quotient} we have seen that $(\Mat_{rN}\times\Mat_{Ns}\times\Mat_N,Y_{r,s,N}\times\mc N_{N,m})$
is a good pair of $\GL_{r,s,N}$-varieties. So by Proposition~\ref{prop.quotient}(i) and \cite[Prop.~1.4(a)]{Don1}\medskip

\noindent(a).\ $(\ov\varphi_{r,s,N,N-1}(\Mat_{rN}\times\Mat_{Ns}\times\Mat_N),\Mat_{rs}^m)$ is a good pair of $\GL_r\times\GL_s$-\hbox{\hspace{.7cm}}varieties.\medskip

Put $Z_{N,n}=\{(B,X)\in\Mat_{Ns}\times\Mat_N\,|\,\rk(B|X)\le n\}$. If we identify $\Mat_{Ns}\times\Mat_N$ with $\Mat_{N,s+N}$, then
$(\Mat_{Ns}\times\Mat_N,Z_{N,n})$ is a good pair of $\GL_N\times\GL_{s+N}$-varieties by \cite[Prop.~1.4(c)]{Don1}.
By \cite[Cor.~4.2.15]{BK} it is then a good pair of $\GL_N\times(\GL_s\times\GL_N)$-varieties
and by \cite[Cor.~4.2.14]{BK} it is then also a good pair of $\GL_s\times\GL_N$-varieties if we let $\GL_N$ act diagonally.
It will also be a good pair of $\GL_{r,s,N}$-varieties if we let $\GL_r$ act trivially.
So by \cite[Prop~1.3e(i)]{Don1} $(\Mat_{rN}\times\Mat_{Ns}\times\Mat_N,\Mat_{rN}\times Z_{N,n})$ is a good pair of $\GL_{r,s,N}$-varieties.
It now follows from \cite[Prop.~1.4a]{Don1} that\medskip

\noindent(b).\ $(\ov\varphi_{r,s,N,N-1}(\Mat_{rN}\times\Mat_{Ns}\times\Mat_N),\ov\varphi_{r,s,N,N-1}(\Mat_{rN}\times Z_{N,n}))$
is a good \hbox{\hspace{.7cm}}pair of $\GL_r\times\GL_s$-varieties.\medskip

Let $(e_1,\ldots,e_N)$ be the standard basis of $k^N$ and let $(A,B,X)\in\Mat_{rN}\times Z_{N,n}$. Then $\dim(\Im(B)+\Im(X))\le n$, so
for some $g\in\GL_N$ we have $$\Im(gB)+\Im(gX)=\Im(gB)+\Im(gXg^{-1})\subseteq\{e_1,\ldots,e_n\}.$$
Write $$g\cdot A=\bmat{A_1&&A_2},\ g\cdot X=\bmat{X_1&&X_2\\0&&0},\ g\cdot B=\bmat{B_1\\0}\,,$$
$A_1\in\Mat_{rn}, X_1\in\Mat_n, B_1\in\Mat_{ns}$. Then a simple computation shows that $\ov\varphi_{r,s,N,N-1}(A,B,X)=\ov\varphi_{r,s,n,N-1}(A_1,B_1,X_1)$, so\medskip

\noindent(c).\ $\ov\varphi_{r,s,N,N-1}(\Mat_{rN}\times Z_{N,n})=\ov\varphi_{r,s,n,N-1}(\Mat_{rn}\times\Mat_{ns}\times\Mat_n)$,\medskip

since the inclusion $\supseteq$ is obvious.
In the proof of Proposition~\ref{prop.quotient} we saw that $(\Mat_{rn}\times\Mat_{ns}\times\Mat_n,Y_{r,s,n}\times\ov{\mc O}_\nu)$ is a good pair
of $\GL_{r,s,n}$-varieties. So by \cite[Prop.~1.4a]{Don1} we have\medskip

\noindent(d).\  $(\ov\varphi_{r,s,n,N-1}(\Mat_{rn}\times\Mat_{ns}\times\Mat_n),\varphi_{r,s,n,m}(Y_{r,s,n}\times\ov{\mc O}_\nu))$ is a good pair
of \hbox{\hspace{.7cm}}$\GL_r\times\GL_s$-varieties.\medskip

Combining (a)-(d) and \cite[Lem.~1.3a(ii)]{Don1} we obtain the assertion.
\end{proof}

\begin{remsgl}
1.\ Similar as in the proof of Proposition~\ref{prop.good_pair}, one can show that for $r$ and $s$ arbitrary
$(\Mat_{rs}\times\Mat_{rs}^m,\varphi_{r,s,n,m}(\Mat_{rn}\times\Mat_{ns}\times\ov{\mc O}_\nu))$ is a good pair
of $\GL_r\times\GL_s$-varieties.\\
2.\ The result \cite[Thm~2.2a(ii)]{Don1} can also be deduced from \cite[Thm.~4.3]{MvdK} in combination with \cite[Ex. 4.2.E.2]{BK}.
The point is that the splitting from \cite{MvdK} is easily seen to be $B$-canonical.
\end{remsgl}

By Proposition~\ref{prop.quotient}(ii) we have $W_{n,m}\nts\cong\nts Y_{r,s,n}\times^{\GL_n}\mc N_{n,m}\nts:=\nts(Y_{r,s,n}\times\mc N_{n,m})/\nts/\GL_n$.
It is well-known that the formal character of $k[Y_{r,s,n}]$ is independent of the characteristic
(this can also be deduced from the formula in \cite[Prop1.3b(ii)]{Don1}). So by \cite[Thm~6.3]{KV} and \cite[Thm~9]{H} (see also \cite[Thm~3.3]{Br})
the sections in a good $\GL_{r,s,n}$-filtration of $k[Y_{r,s,n}]$ are precisely the induced $\GL_{r,s,n}$-modules
$\nabla_{\GL_r}(-\mu^{\rm rev})\ot\nabla_{\GL_s}(\lambda)\ot\nabla_{\GL_n}([\mu,\lambda])$, each occurring once,
where $\lambda$ and $\mu$ are partitions with $l(\mu)\le r$ and $l(\lambda)\le s$.

Now if $V$ is a good $\GL_n$-variety, then $Y_{r,s,n}\times^{\GL_n}V$ is a good $\GL_r\times\GL_s$-variety by \cite[Prop~1.2e(iii)]{Don1}
and, by the above and a simple character calculation, the good filtration multiplicity of
$\nabla_{\GL_r}(-\mu^{\rm rev})\ot\nabla_{\GL_s}(\lambda)$ in $k[Y_{r,s,n}\times^{\GL_n}V]$ is equal to that of $\nabla_{\GL_n}([\lambda,\mu])$ in $k[V]$.
Note here that $\nabla_{\GL_n}([\mu,\lambda])^*\cong\Delta_{\GL_n}([\lambda,\mu])$.
Loosely spoken, each copy of $\nabla_{\GL_n}([\lambda,\mu])$ in $k[V]$ is replaced by $\nabla_{\GL_r}(-\mu^{\rm rev})\ot\nabla_{\GL_s}(\lambda)$
if $l(\mu)\le r$ and $l(\lambda)\le s$ and removed otherwise.
We can apply this to $V=\mc N_{n,m}$.

If we give the piece of $k[\Mat_{rs}^m]$ of multidegree $\nu$ total degree $\sum_{i=1}^m\nu_ii$, then the vanishing ideals of the varieties $W_{n,m}$ are graded,
so their coordinate rings will inherit the above total grading. The aforementioned
equalities of good filtration multiplicities for $k[\mc N_{n,m}]$ and $k[W_{n,m}]$ are then in fact equalities of graded good filtration multiplicities.
Furthermore, the graded dimension of $k[\mc N_{n,m}]^{U_n}_{[\lambda,\mu]}$ is increasing in $m$, and by the above it is also increasing in $n$,
since $W_{n,m}\subseteq W_{N,m}$ whenever $N\ge n$. It follows that the graded dimension of $k[\mc N_n]^{U_n}_{[\lambda,\mu]}$ is increasing in $n$.
This was observed by R.~Brylinsky in \cite{Br}.

The theorem below says that to find finite spanning sets for the highest weight vectors in the coordinate ring of the $\GL_n$-variety $\mc N_{n,m}$,
it is enough to do this for the $\GL_r\times\GL_s$-variety $\Mat_{rs}^m$. We note that, since $k[\Mat_{rs}^m]$ has a good filtration
and its formal character is independent of the characteristic, the good filtration multiplicity $\dim k[\Mat_{rs}^m]^{U_r\times U_s}_{(-\mu^{\rm rev},\lambda)}$
of $\nabla_{\GL_r}(-\mu^{\rm rev})\ot\nabla_{\GL_s}(\lambda)$ in $k[\Mat_{rs}^m]$ is independent of the characteristic of $k$.
A simple character calculation combined with \cite[Ex.~I.7.10(b)]{Mac} shows that the multigraded good filtration multiplicity of
$\nabla_{\GL_r}(-\mu^{\rm rev})\ot\nabla_{\GL_s}(\lambda)$ in $k[\Mat_{rs}^m]$ is $s_\lambda\ast s_\mu(z_1,\ldots,z_m)$,
where $s_\lambda$ is the Schur function associated to $\lambda$, $\ast$ denotes the internal product of Schur functions and $z_i$ is
a grading variable for the $i$-th matrix component. So this multiplicity is $0$ if $|\lambda|\ne|\mu|$ or if $s_\lambda\ast s_\mu$
only contains Schur functions associated to partitions of length $>m$.

\begin{thmgl}\label{thm.surjective_pullback}
Let $\chi=[\lambda,\mu]$ be a dominant weight in the root lattice, $l(\mu)\le r$, $l(\lambda)\le s$, $r+s\le n$,
and let $\nu$ be a partition of $n$ with $\nu_1\le m+1$.
Then the pull-back $$k[\Mat_{rs}^m]^{U_r\times U_s}_{(-\mu^{\rm rev},\lambda)}\to k[\ov{\mc O}_\nu]^{U_n}_\chi$$
along $\varphi_{r,s,n,m}:\ov{\mc O}_\nu\to\Mat_{rs}^m$ is surjective, and in case $\ov{\mc O}_\nu=\mc N_{n,m}$ and
$n\ge (m+1)\max(r,s)$ it is an isomorphism.
\end{thmgl}
\begin{proof}
For a matrix $M$ denote by $M_{r\rfloor,\lfloor s}$ the lower left $r\times s$ corner of $M$ and define $M_{r\rfloor,r\rfloor}$ and $M_{\lfloor s,\lfloor s}$ similarly. Then we have
$$(SXS^{-1})_{r\rfloor,\lfloor s}=S_{r\rfloor,r\rfloor}X_{r\rfloor,\lfloor s}(S_{\lfloor s,\lfloor s})^{-1}$$
and therefore 
$$\varphi_{r,s,n,m}(SXS^{-1})=S_{r\rfloor,r\rfloor}\varphi_{r,s,n,m}(X)(S_{\lfloor s,\lfloor s})^{-1}$$ for any $X\in\Mat_n$ and any upper triangular $S\in\GL_n$.
So indeed the pull-back along $\varphi_{r,s,n,m}$ maps highest weighty vectors to highest weight vectors and it is an easy exercise to see that
the weights correspond as stated in the theorem.

Since $(\mc N_{n,m},\ov{\mc O}_\nu)$ is a good pair of $\GL_n$ varieties by \cite[Thm.~2.1c, Lem.~1.3a(ii)]{Don1} we may assume $\ov{\mc O}_\nu=\mc N_{n,m}$.
By the discussion before the theorem, based on Proposition~\ref{prop.quotient}, we know that the good filtration multiplicity of $\nabla_{\GL_r}(-\mu^{\rm rev})\ot\nabla_{\GL_s}(\lambda)$
in $k[W_{nm}]$ is equal to that of $\nabla_{\GL_n}([\lambda,\mu])$ in $k[\mc N_{n,m}]$.
Put differently, we know that $k[W_{nm}]^{U_r\times U_s}_{(-\mu^{\rm rev},\lambda)}$ and $k[\mc N_{n,m}]^{U_n}_\chi$ have the same dimension.
As we have seen before, $\varphi_{r,s,n,m}(\mc N_{n,m})$ is dense in $W_{nm}$, so the pull-back $k[W_{nm}]\to k[\mc N_{n,m}]$ along $\varphi_{r,s,n,m}$
is injective and induces an isomorphism between $k[W_{nm}]^{U_r\times U_s}_{(-\mu^{\rm rev},\lambda)}$ and $k[\mc N_{n,m}]^{U_n}_\chi$.
By the argument in the proof of \cite[Thm~5.1]{Br} we have $\varphi_{r,s,n,m}(\mc N_{n,m})=\Mat_{rs}^m$ if $n\ge (m+1)\max(r,s)$, which gives us the final assertion.
So it suffices to show that the restriction $k[\Mat_{rs}^m]^{U_r\times U_s}_{(-\mu^{\rm rev},\lambda)}\to k[W_{nm}]^{U_r\times U_s}_{(-\mu^{\rm rev},\lambda)}$
is surjective and this follows from Proposition~\ref{prop.good_pair}.
\end{proof}
\begin{remsgl}\label{rems.surjective_pullback}
1.\ Assume $m=n-1$. If $f\in k[\Mat_{rs}^m]^{U_r\times U_s}_{(-\mu^{\rm rev},\lambda)}$ is homogeneous for the total grading defined above, then the pull-back of $f$ along $\varphi_{r,s,n,m}:\mc N_n\to\Mat_{rs}^m$ has an obvious lift to $k[\Mat_n]^{U_n}_{[\lambda,\mu]}$, namely the pull-back of $f$ along $\varphi_{r,s,n,m}:\mc \Mat_n\to\Mat_{rs}^m$. This follows from the fact that the displayed formulas at the beginning of the proof of Theorem~\ref{thm.surjective_pullback} hold for any $X\in\Mat_n$. So if we have a spanning set of $k[\mc N_n]^{U_n}_{[\lambda,\mu]}$ which is pulled back from $k[\Mat_{rs}^m]^{U_r\times U_s}_{(-\mu^{\rm rev},\lambda)}$ along $\varphi_{r,s,n,m}$, then we are always in the situation to apply Lemma~\ref{lem.reduction_to_nilpotent_cone}.\\
2.\ With the total grading of $k[\Mat_{rs}^m]$ defined above the pull-back along $\varphi_{r,s,n,m}:\mc N_n\to\Mat_{rs}^m$ is a homomorphism of graded vector spaces. By \cite[Thm.~2.14]{BCHLLS} and the independence of the characteristic of the graded formal character of $k[\mc N_n]$, the good filtration multiplicity of $\nabla_{\GL_n}([\lambda,\mu])$ in the degree $d$ piece of $k[\mc N_n]$ is the same for all $n\ge l(\lambda)+l(\mu)+d-t$, where $t=|\lambda|=|\mu|$.
From this, Theorem~\ref{thm.surjective_pullback} and the fact that the graded dimension of $k[\mc N_{n,m}]^{U_n}_{[\lambda,\mu]}$ is increasing in $m$ and $n$ it follows that
the  pull-back $k[\Mat_{rs}^{n-1}]^{U_r\times U_s}_{(-\mu^{\rm rev},\lambda)}\to k[\mc N_n]^{U_n}_{[\lambda,\mu]}$ will be an isomorphism in degree $d$ if $n\ge l(\lambda)+l(\mu)+d-t$.\\
\end{remsgl}

The space $\Mat_{rs}^m=\Mat_{rs}\ot k^m$ has an extra action of the group $\GL_m$ which commutes with the action of $\GL_r\times\GL_s$.
For convenience we choose the action induced by the action $g\cdot v=vg^{-1}$ on $k^m$, where $v$ is considered as a row vector.
If we would have used the more obvious action $g\cdot v=gv$ on $k^m$, then this would amount to twisting the above action with the inverse transpose. 

Let $\lambda$ be a partition of $t\le r$ with $l(\lambda)\le s$. 
For $T$ a tableau of shape $\lambda$ with entries $\le m$ we define the semi-invariant $u_T\in k[\Mat_{rs}^m]$ by
\begin{align*}
&(A_1,\ldots,A_m)\mapsto\\
\sum_S\det\big(A_{S_{11}}e_1|\cdots|&A_{S_{1\lambda_1}}e_1|\cdots|A_{S_{l(\lambda)1}}e_{l(\lambda)}|\cdots|
A_{S_{l(\lambda)\lambda_{l(\lambda)}}}e_{l(\lambda)}\big)_{t\rfloor}
\end{align*}
and the semi-invariant $v_T\in k[\Mat_{sr}^m]$ by
\begin{align*}
&(A_1,\ldots,A_m)\mapsto\\
\sum_S\det\big(A_{S_{11}}'e_s|\cdots|A_{S_{1\lambda_1}}'&e_s|\cdots|A_{S_{l(\lambda)1}}'e_{s-l(\lambda)+1}|\cdots|
A_{S_{l(\lambda)\lambda_{l(\lambda)}}}'e_{s-l(\lambda)+1}\big)_{\lfloor t}\,,
\end{align*}
where the sums are over all tableaux $S$ in the orbit of $T$ under the column stabiliser $C_\lambda\le\Sym(\lambda)$ of $\lambda$,
the subscripts ``$t\rfloor$" and ``$\lfloor t$" mean that we take the last resp. first $t$ rows, the $S_{ij}$ denote the entries of $S$,
the $e_i$ are the standard basis vectors of $k^s$, and $A_i'$ denotes the transpose of $A_i$.


\begin{thmgl}\label{thm.basis_special_weights} 
Let $\lambda$ be a partition of $t\le r$ with $l(\lambda)\le s$ and $\lambda_1\le m$. Then
\begin{enumerate}[{\rm (i)}]
\item the $u_T$ with $T$ row semi-standard form a basis of $k[\Mat_{rs}^m]^{U_r\times U_s}_{(-(1^t)^{\rm rev},\lambda)}$\,,
\item the $v_T$ with $T$ row semi-standard form a basis of $k[\Mat_{s\,r}^m]^{U_s\times U_r}_{(-\lambda^{\rm rev},1^t)}$\,,
\end{enumerate}
and both vector spaces are, with the $\GL_m$-action defined above, isomorphic to the Weyl module of highest weight $\lambda'$.
\end{thmgl}
\begin{proof}
(i).\ Put $F=k^m$,  let $(f_1,\ldots,f_m)$ be the standard basis of $F$ and put
$\bigwedge^{\lambda}F=\bigwedge^{\lambda_1}F\ot\cdots\ot\bigwedge^{\lambda_{l(\lambda)}}F$. For $S$ a tableau of shape $\lambda$ with entries $\le m$
we put $$f_S=f_{S_{11}}\wedge\cdots\wedge f_{S_{1\lambda_1}}\ot\cdots\ot f_{S_{l(\lambda)1}}\wedge\cdots\wedge f_{S_{l(\lambda)\lambda_{l(\lambda)}}}\,.$$
Then the $f_S$ with the rows of $S$ strictly increasing form a basis of $\bigwedge^{\lambda}F$.
From the anti-symmetry properties of the $f_S$ it is clear that there exists a unique linear mapping
$\psi:\bigwedge^{\lambda}F\to k[\Mat_{rs}^m]$ with\quad $\psi(f_S)=$
$$(A_1,\ldots,A_m)\mapsto\det\big(A_{S_{11}}e_1|\cdots|A_{S_{1\lambda_1}}e_1|\cdots|A_{S_{l(\lambda)1}}e_{l(\lambda)}|\cdots|
A_{S_{l(\lambda)\lambda_{l(\lambda)}}}e_{l(\lambda)}\big)_{t\rfloor}$$
for all tableaux $S$ of shape $\lambda$ with entries $\le m$. Furthermore, it is easy to check that $\psi$ is $\GL_m$-equivariant
and that the $u_T$, $T$ row semi-standard are the images of the Carter-Lustig basis elements of the Weyl module of highest weight $\lambda'$
inside $\bigwedge^{\lambda}F$, see \cite[5.3b]{Gr} and \cite[Thm 3.5]{CL}. 
So to prove (i) and the final assertion in case (i) it suffices to show that $\psi$ is injective
and $k[\Mat_{rs}^m]^{U_r\times U_s}_{(-(1^t)^{\rm rev},\lambda)}$ has dimension equal to that of the Weyl module of highest weight $\lambda'$.
Since the space of highest weight vectors has dimension $s_{1^t}\ast s_\lambda(1,\ldots,1)=s_{\lambda'}(1,\ldots,1)$ ($m$ ones) the latter is
indeed true, so it remains to prove the injectivity of $\psi$.

To prove this will associate to each tableau $T$ of shape $\lambda$ with entries $\le m$ and strictly increasing rows an $m$-tuple of
$r\times s$-matrices $A(T)$ such that $\psi(f_S)(A(T))_{S,T}$ is the identity matrix. We define $A(T)$ as follows
$$A(T)_{T_{ij}}(e_i)=e_{(T_\lambda)_{ij}}\text{\quad and\quad} A(T)_h(e_i)=0\text{\ if\ } h\notin \text{$i$-th row of $T$ or $l(\lambda)<i\le s$}\,,$$
where $T_\lambda$ is the tableau of shape $\lambda$ defined in Section~\ref{s.prelim}, and we denote the standard basis vectors of $k^{\max(r,s)}$ by
$e_1,\ldots,e_{\max(r,s)}$.\footnote{The reader may consider $k^r$ as a subspace of $k^s$ if $r\le s$ and conversely otherwise.}
Then clearly $\psi(f_T)(A(T))=1$. Now assume $S\ne T$. Then $S_{ij}\ne T_{ij}$ for certain $i,j$, so $S_{ij}$ does not occur in the $i$-th row
of $T$. So $A(T)_{S_{ij}}(e_i)=0$ and therefore $\psi(f_S)(A(T))=0$.\\
(ii).\ Let $\Phi:k[\Mat_{rs}^m]\to k[\Mat_{s\,r}^m]$ be the algebra isomorphism corresponding to vector space isomorphism $\Mat_{s\,r}^m\to \Mat_{rs}^m$
induced by the vector space isomorphism $A\mapsto P_1A'P_2^{-1}:\Mat_{s\,r}\to\Mat_{rs}$, where $P_1\in\GL_r$ and $P_2\in\GL_s$
are the permutation matrices which are $1$ on the anti-diagonal and $0$ elsewhere. Then
$\Phi(k[\Mat_{rs}^m]^{U_r\times U_s}_{(-(1^t)^{\rm rev},\lambda)})=k[\Mat_{s\,r}^m]^{U_s\times U_r}_{(-\lambda^{\rm rev},1^t)}$ and $\Phi(u_T)=\pm v_T$.
So (ii) follows from (i). Furthermore, $\Phi$ is $\GL_m$-equivariant, so the final assertion also applies to (ii).
\end{proof}

\begin{remsgl}\label{rems.basis_special_weights}
1.\ If $\lambda$ or $\mu$ is a row one can easily find bases of $k[\Mat_{s\,r}^m]^{U_s\times U_r}_{(-\mu^{\rm rev},\lambda)}$. In this case the $\GL_m$-module structure is that of the induced module of highest weight $\lambda$. Unlike the case that $\lambda$ or $\mu$ is a column, the pull-backs of these bases to the nilpotent cone are always bases of $k[\mc N_n]^{U_n}_{[\lambda,\mu]}$. This can be deduced from the proof of \cite[Thm.~2]{T2}. For example, for the weight $(-\lambda^{\rm rev},(t))$, $l(\lambda)\le m$, one obtains a basis by taking the ``left anti-canonical bideterminants" $(\tilde T_\lambda\,|\,T)$, $T$ semi-standard of shape $\lambda$ with entries $\le m$, on the $r\times m$ matrix obtained by taking the first column of each matrix component of $\un A\in\Mat_{rs}^m$. Here $\tilde T_\lambda$ is the anti-canonical tableau denoted by $T_\lambda$ in \cite{T2}. Our results on the $\GL_m$-module structure when $\lambda$ or $\mu$ is a row or a column are in accordance with \cite[Sect.~III]{ABW}.\\
2.\ Combining Theorem~\ref{thm.basis_special_weights} and Theorem~\ref{thm.surjective_pullback} we obtain spanning sets for the spaces $k[\mc N_n]^{U_n}_\chi$, where $\chi$ is of the form $[\lambda,1^t]$ or $[1^t,\lambda]$, i.e. for weights $\chi$ with $\chi_{{}_n}\ge-1$ or with $\chi_{{}_1}\le1$. Assume ${\rm char}(k)=0$. Then the weights $\chi$ with $\chi_{{}_n}\ge-1$ are related to the coinvariant ring $C_W$ of $W=\Sym_n$ via the generalised Chevalley Restriction Theorem as follows:
$$k[\mc N_n]^{U_n}_\chi\cong\Mor_{\GL_n}(\mc N_n,L(\chi)^*)\cong\Mor_W(\mc N_n\cap\t,L(\chi)^*_0)\cong\Hom_W(L(\chi)_0,C_W)\,.$$
Here $\mc N_n\cap\t$ is the scheme-theoretic intersection of $\mc N_n$ and the vector space of diagonal $n\times n$-matrices $\t$. In fact one can replace $\mc N_n$ by an arbitrary nilpotent orbit closure $\ov{\mc O}_\nu$ and $C_W$ by the corresponding coinvariant ring, see \cite{Broer}. This means in particular that the graded dimension of $k[\ov{\mc O}_\nu]^{U_n}_\chi$ is given by $\tilde K_{\ov\lambda',\nu'}(t)$, where $\ov\lambda=\chi+{\bf 1}_n$, ${\bf 1}_n$ the all-one vector of length $n$ and $\tilde K_{\ov\lambda',\nu'}(t)=t^{n(\nu')}K_{\ov\lambda',\nu'}(t^{-1})$, $K_{\ov\lambda',\nu'}(t)$ the Kostka polynomial, as in \cite[p.~248]{Mac}, see e.g. \cite{GP}.\\
3.\ For weights of the form $[\lambda,1^t]$, $[1^t,\lambda]$, $[t,\lambda]$ and $[\lambda,t]$ the dimension of the lowest degree piece is always one. In the first case this follows from the link with the coinvariant algebra mentioned above (take $\nu=(n)$). In the second case this follows from the well-known connection with Kostka polynomials, see \cite[p2, Rem.~2.2]{T1}.  In general it need not be true: for $\chi=(3,3,0,-2,-2,-2)$, the lowest degree of $k[\mc N_n]^{U_n}_\chi$ is 9 and the piece of degree 9 has dimension 2. 

By going to bigger $n$ the lowest degree of $k[\mc N_n]^{U_n}_{[\lambda,\mu]}$ may drop: for $\lambda=(4,4,4)$ and $\mu=(3,3,3,3)$ the lowest degree is 18 for $n=7$ and 17 for $n=8$.
All this can be calculated with the computer using the Lascoux-Sch\"utzenberger charge on tableaux \cite{LS}.
\end{remsgl}

\section{Coinvariants for Young subgroups and highest weight vectors in characteristic $0$}\label{s.char0}
In this section we want to give bases for all the spaces of highest weight vectors in $k[\Mat_{rs}^m]$.
We will always assume that $k$ has characteristic $0$.
\subsection{Representations of the symmetric group}\label{ss.Sym}
We give a short account of Donin's results \cite{Donin1} on the representations of the symmetric group.
He gave certain explicit bases for Hom spaces between skew Specht modules which are useful for the
purpose of finding natural spanning sets for the highest weight vectors in $k[\mc N_n]$.
We drop the assumption that $k$ is algebraically closed.
Let $G$ be a finite group and let $A=kG$ be its group algebra. It has the obvious $\mb Q$-form $A_{\mb Q}=\mb QG$.
Denote the symmetric bilinear form on $A$ for which the group elements form an orthonormal basis by $(-,-)$.
Since its restriction to $A_{\mb Q}$ is positive definite, its restriction to any $\mb Q$-defined subspace of $A$ will be nondegenerate.
Let $a\mapsto a^*$ be the anti-involution of $A$ which extends the
inversion of $G$. Then we have $$(ab,c)=(a,cb^*)\text{\ and\ }(ab,c)=(b,a^*c)$$ for all $a,b,c\in A$.
To deal with Hom spaces between ideals of $A$ generated by elements that need not be idempotents we need the following lemma.
\begin{lemgl}\label{lem.homspaces}
Let $a\in A$ and let $M$ be an $A$-module.
\begin{enumerate}[{\rm(i)}]
\item The map $\varphi:x\ot y\mapsto x^*y:Aa\ot M\to a^*M$ restricts to an isomorphism $(Aa\ot M)^G\stackrel{\sim}{\to} a^*M$.
The inverse is given by $\psi:c\mapsto\frac{1}{|G|}\sum_{g\in G}g\ot gc$.
\item If $a\in A_{\mb Q}$, then the composite of $\psi$ with the $G$-module isomorphism $x\ot y\mapsto(z\mapsto(x,z) y):Aa\ot M\to\Hom(Aa,M)$ maps $c\in\nts a^*M$
to the ``right multiplication" by $\frac{1}{|G|}c$.
\item If $a\in A_{\mb Q}$, then $Aa=Aa^*a$.
\end{enumerate}
\end{lemgl}
\begin{proof}
(i).\ Clearly, $\varphi\circ\psi=\id$. Furthermore, we have for all $x,y\in A$ and $z\in M$
$$\sum_{g\in G}gxy\ot gz=\sum_{g\in G}gy\ot gx^*z\,.$$
So if $x\in a^*M$, then $\psi(x)\in(Aa\ot M)^G$. Now $(Aa\ot M)^G$ is spanned
by elements of the form $c=\sum_{g\in G}gxa\ot gy$, $x\in A$, $y\in M$, and for such a $c$ we have
$\psi(\varphi(c))=\psi(|G|(xa)^*y)=\sum_{g\in G}g\ot g(xa)^*y=\sum_{g\in G}gxa\ot gy=c$.\\
(ii).\ First note that the given map from $Aa\ot M$ to $\Hom(Aa,M)$ is obtained by combining the standard isomorphism
$(Aa)^*\ot M\stackrel{\sim}{\to}\Hom(Aa,M)$ with the isomorphism $x\mapsto(x,-):Aa\stackrel{\sim}{\to}(Aa)^*$,
so it is indeed an isomorphism. Now we compose $\psi$ with this isomorphism.
Then $c\in a^*M$ goes to the map $z\mapsto\frac{1}{|G|}\sum_{g\in G}(g,z)gc=z\frac{1}{|G|}c$.\\
(iii).\ Let $\rho_a$ denote the right multiplication by $a$. Then $\rho_{a^*}=\rho_a'$, the transpose of $\rho_a$ with respect to the form $(-,-)$.
So $Aa^*a=\Im(\rho_a\rho_a')=\Im(\rho_a)=Aa$. Here the second equality follows from the corresponding equality on $A_{\mb Q}$
on which our form is positive definite.
\end{proof}

From now on $G$ will be the symmetric group $\Sym_t$ of rank $t$. To describe certain Hom spaces and
certain subspaces of $A$ it will turn out to be useful to use bijections between skew diagrams.
We call such bijections {\it diagram mappings}. If we fix skew diagrams $E$ and $F$,
then the elements of $G$ are in one-one correspondence with diagram mappings $F\to E$ as follows.
If $\alpha:F\to E$ is a diagram mapping, then the corresponding element of $G$ sends for any box $x$ of $F$ the number
of $T_F$ in $x$ to the number of $T_E$ in $\alpha(x)$. If we fix only one skew diagram $E$,
then we can identify the elements of $G$ with $t$-tableaux of shape $E$ by replacing $(E,F)$ above by $(\Delta_t,E)$
and use the fact that $t$-tableaux can be identified with diagram mappings $E\to \Delta_t$.
So the first correspondence is $g\mapsto\alpha_g=T_E^{-1}\circ g\circ T_F$ and the second one is $g\mapsto g\circ T_E$.
For $T$ a $t$-tableau of shape $F$ we will also denote $T_E^{-1}\circ T$ by $\alpha_T$.

As is well known one can associate the so-called skew Specht modules to skew diagrams, just like
one can associate Specht modules to ordinary Young diagrams. These skew Specht modules are in general not irreducible,
in fact they include the Young permutation modules. We briefly recall the construction.
If $E$ is a skew Young diagram with $t$ boxes, then we can form the row symmetriser $e_2=\sum_gg\in A_{\mb Q}$
where the sum is over the row stabliser of $T_E$ in $G$, and the column anti-symmetriser $e_1=\sum_g{\rm sgn}(g)g\in A_{\mb Q}$ where
the sum is over the column stabiliser of $T_E$ in $G$. The product $e=e_1e_2$ is then called the {\it Young symmetriser}
associated to the skew diagram $E$. Unlike in the case of ordinary Young diagrams, the symmetrisers associated to skew diagrams
are no longer idempotent up to a scalar multiple, although $e_1$ and $e_2$ of course are. \hbox{\hspace{2cm}}\vspace{-2mm}

\noindent For example, if
$\ytableausetup
{mathmode, boxsize=1.3em}
E={\begin{ytableau}
	\none&1&2\\
	3&4
	\end{ytableau}}\ ,$
\vspace{3mm}
then $\dim{\rm span}(e,e^2)=2$.

The {\it skew Specht module} associated to $E$ is the module $Ae$. We have $Ae=Ae_1e_2\subseteq Ae_2$
and $Ae_2$ is the well-known permutation module associated to $E$.  If $\lambda$ is the partition which
contains the row lengths of $E$ in weakly descending order, then $Ae_2$ is isomorphic to the usual Young permutation module $M^\lambda$.
For example, if $\lambda$ is a partition of length $l$ and
$$\begin{xy}
	(0.3,-.5)*=<48pt,28pt>{}*\frm{^\}},
	(0.3,7.5)*{\text{$\lambda_1$ boxes}},
	(-20,-10)*=<48pt,28pt>{}*\frm{_\}},
	(-20,-18)*{\text{$\lambda_l$ boxes}},
	(-9.9,-5.2)*={\begin{ytableau}\none\\\none[E=\quad\quad]\\ \none\end{ytableau}
\begin{ytableau}
\none&\none&\none&\none&\ &\none[\cdots]&\ \\
\none&\none&\none&\none[\iddots]&\none&\none&\none\quad,\\ 
\ &\none[\cdots]&\ &\none&\none&\none&\none
\end{ytableau}}
\end{xy}$$
then $e=e_2$ and $Ae=Ae_2=M^\lambda$. If $g,h\in G$, then $ge_2=he_2$ if and only if the tableaux of shape $E$
corresponding to $g$ and $h$ are row equivalent. For $g\in G$ and $T=g\circ T_E$ we
denote $ge_2$ by $\{T\}$ and call it a {\it tabloid} in accordance with \cite{J}. Furthermore, $ge=ge_1g^{-1}ge_2$ and $\kappa_T=ge_1g^{-1}$
is the column anti-symmetriser associated to the skew tableau $T$.
So the element $ge=\kappa_T\{T\}$ is the {\it polytabloid} $e_T$ from \cite{J}. We will denote it by $[T]$. For a $t$-tableau $T$ of shape $E$ we have
$[T]=\sum_{\pi\in C_E}{\rm sgn}(\pi)\{T\pi\}$, where $C_E\le\Sym(E)$ is the column stabiliser of $E$.

For the remainder of this section $E$ and $F$ are two skew diagrams and
$e=e_1e_2$ and $f=f_1f_2$ are the corresponding Young symmetrisers.
The next lemma says that, just like Specht modules, skew Specht modules could also have been defined by multiplying
row symmetrisers and column anti-symmetrisers the other way round.
\begin{lemgl}\
\begin{enumerate}[{\rm(i)}]
\item We have $Ae_1e_2=Ae_2e_1e_2$ and $Ae_2e_1=Ae_1e_2e_1$.
\item The maps $x\mapsto xe_1:Ae_1e_2\to Ae_2e_1$ and $x\mapsto xe_2:Ae_2e_1\to Ae_1e_2$ are isomorphisms.
\end{enumerate}
\end{lemgl}
\begin{proof}
(i).\ Since $e_1^*=e_1$ and $e_2^*=e_2$, we have $e^*=e_2e_1$ and $e^*e$ is a nonzero scalar multiple of $e_2e_1e_2$.
Similarly for $\tilde e=e_2e_1$ we have that $\tilde e^*\tilde e$ is a nonzero scalar multiple of $e_1e_2e_1$.
The assertion now follows from Lemma~\ref{lem.homspaces}(iii).\\
(ii).\ By (i) these maps are surjective, so, for dimension reasons, they must be isomorphisms.
\end{proof}
Since the elements of $G$ can be considered as diagram mappings $:F\to E$ we get a spanning set of
$\Hom_A(Ae,Af)=e^*Af$ which is labelled by diagram mappings $:F\to E$.
In particular we think of $Ae$ as spanned by diagram mappings $:E\to\Delta_t$, i.e. $t$-tableaux of shape $E$.
It is our goal to find a subset of the above spanning set which is a basis for the space $e^*Af$.
First we point out some special cases, then we state it in general in Theorem~\ref{thm.homspace_basis}.
Let $\mu$ be the tuple of row lengths of $E$, i.e. the weight of $S_E$. We have for $g,h\in G$ that $e_2g=e_2h$ if and only if $S_E\circ\alpha_g=S_E\circ\alpha_h$.
We will say that $g$ or $T=g\circ T_F$ or $\alpha_T=T_E^{-1}\circ T$ {\it represents} $S_E\circ\alpha_g=S_E\circ\alpha_T$. So the elements $e_2g$ with $g$ in a set of representatives for the
tableaux of shape $F$ and weight $\mu$ form a basis of $e_2A$. Of course we could change the shape $F$ to any other shape with the same number of boxes.
More generally, we have for $T_1,T_2$ $t$-tableaux of shape $F$ that $e_2\{T_1\}=e_2\{T_2\}$ if and only if $S_E\circ\alpha_{T_1}$ and $S_E\circ\alpha_{T_2}$
are row equivalent.
So the elements $e_2\{T\}$ with $T$ in a set of representatives for the row-ordered tableaux of
shape $F$ and weight $\mu$ form a basis of $e_2Af_2$.
For a tableau $T$ we define the {\it standard scan} of $T$ to be the sequence of entries of $T$, read row by row
from left to right and top to bottom.
We order the row ordered tableaux of shape $F$ as follows. If $S\ne T$ are two such tableaux, then $S<T$ if and only if
$\alpha_i<\beta_i$, where $i$ is the first position where the standard scans $\alpha$ and $\beta$ of $S$ and $T$ differ. 
The above basis of $e_2Af_2$ is now also linearly ordered, since we linearly ordered its index set.
We extend the above order to a preorder on all tableaux of shape $F$ by defining $S\le T$ if and only if $\tilde S\le\tilde T$,
where $\tilde S$ and $\tilde T$ are the unique row ordered tableaux that are row-equivalent to $S$ resp. $T$.
The proof of the next trivial lemma is left to the reader.

\begin{lemgl}[{cf.~\cite[Lem~1.2]{Donin1}, \cite[Lem.~8.2]{J}}]\label{lem.independent}
Let $(x_i)_{i\in I}$ be a family of elements of $e_2Af_2$ and for each $i$ let $y_i$ be the least element from the above
basis of $e_2Af_2$ involved in $x_i$. If the $y_i$ are distinct, then $(x_i)_{i\in I}$ is linearly independent.
\end{lemgl}

\begin{lemgl}[{cf.~\cite{Donin1}, \cite[Lem.~8.3]{J}}]\label{lem.order}
Let $F$ be a skew diagram. If $S,T$ are distinct column equivalent tableaux of shape $F$ with $S$ column ordered, then $S<T$.
\end{lemgl}
\begin{proof}
Denote the $i$-th rows of $S$ and $T$ by $S_i$ and $T_i$. Choose $i$ minimal with $S_i\ne T_i$.
Then we have $S_{ij}\le T_{ij}$ for all $j$ with at least one inequality strict. So for each $r$ the number of occurrences
of integers $\le r$ in $S_i$ is $\ge$ to that of $T_i$ with at least one inequality strict. So $S<T$.
\end{proof}

As in \cite[Thm.~8.4]{J} one can use the previous two lemma's (replace $(E,F)$ by $(\Delta_t,E)$) and an obvious generalisation
of the Garnir relations \cite[Sect.~7]{J}
to prove the well-known result that the polytabloids $[T]$, $T$ a standard tableau of shape $E$, form a basis of $Ae$.

\begin{lemgl}[{\cite[Lem~2.2]{Donin1} and \cite[Prop.]{Donin2}}]\label{lem.special}
Let $\alpha:F\to E$ be a diagram mapping which satisfies
\begin{enumerate}[{\rm(a)}]
\item The tableau $S_E\circ\alpha$ of shape $F$ is semi-standard.
\item If for $a,b\in F$, $\alpha(b)$ occurs strictly below $\alpha(a)$ in the same column,
then $b$ occurs in a strictly lower row than $a$.
\end{enumerate}
Then there exists a diagram mapping $\tilde\alpha:F\to E$ with $S_E\circ\tilde\alpha=S_E\circ\alpha$ satisfying
\begin{enumerate}[{\rm(b')}]
\item  If for $a,b\in F$, $\tilde\alpha(b)$ occurs strictly below $\tilde\alpha(a)$ in the same column, then $b$ occurs in
a strictly lower row than $a$ and in a column to the left of $a$ or in the same column.
\end{enumerate}
\end{lemgl}
\begin{proof}
Let $a=(i,j)\in F$ be the first cell in the order of the standard scan such that with $\alpha(a)=(r,s)$
we have $(r+1,s)\in E$ and $b=\alpha^{-1}(r+1,s)$ occurs in a column strictly to the right of $a$ (*).
Since $S_E(\alpha(a))=r$, $S_E(\alpha(b))=r+1$, $S_E\circ\alpha$ is semi-standard and $\alpha$ has property (b)
we have $b=(i+1,j_1)$ for some $j_1>j$, $S_E(\alpha(i,j_2))=r$ and $S_E(\alpha(i+1,j_2))=r+1$ for all $j_2$
with $j\le j_2\le j_1$. Now put $b_1=(i+1,j)$ and $\beta=\alpha\circ(b,b_1)$, where $(b,b_1)$ is the transposition
which swaps $b$ and $b_1$. Then $S_E\circ\beta=S_E\circ\alpha$. If $\beta$ does not have property (b'), then the
first cell of $F$ in the order of the standard scan that has property (*) for $\beta$ will be after $a$.
This is clear if with $\alpha(b_1)=(r+1,s_1)$ we have $(r,s_1)\notin E$. So assume this is not the case and
assume $a_1=\alpha^{-1}(r,s_1)$ occurs before $a$ in the standard scan. Then, by the choice of $a$, its column
index is $>j$. So its row index is $<i$. But then, by the semi-standardness of $S_E\circ\alpha$, its column
index is $>j_1$. So $a_1$ doesn't have the above property for $\beta$ and this was the only possibility before $a$. 
So we can finish by induction.
\end{proof}

Recall that $\mu$ is the tuple of row lengths of $E$. We will call a semi-standard tableau $S$ of shape $F$ and
weight $\mu$ {\it special} if $S=S_E\circ\alpha$ for some diagram mapping $\alpha:F\to E$ satisfying the conditions
(a) and (b) from Lemma~\ref{lem.special}. We will call $\alpha$ and $T=T_E\circ\alpha$ {\it admissible} if
$\alpha$ satisfies 
(b'). So, by Lemma~\ref{lem.special}, every special semi-standard tableau of 
shape $F$ and weight $\mu$ has an admissible representative $T$. From now on we will always assume that representatives of special
semi-standard tableaux are admissible.

Next we need the notion of a ``picture" (we will call it special) from
\cite{Z1} which is a generalisation of that of \cite{JP}. For this we need two orderings $\le$ and $\preceq$ on $\mb N\times\mb N$
defined by $(p,q)\le(r,s)$ if and only if $p\le r$ and $q\le s$, and $(p,q)\preceq(r,s)$ if and only if $p<r$ or ($p=r$ and $q\ge s$).
Note that $\preceq$ is a linear ordering. Recall that skew Young diagrams are by definition subsets of $\mb N\times\mb N$.
A diagram mapping $\alpha:F\to E$ is called {\it special} if $\alpha:(F,\le)\to(E,\preceq)$ and 
$\alpha^{-1}:(E,\le)\to(F,\preceq)$ are order preserving. So $\alpha$ is special if and only if $\alpha^{-1}$ is special.
In \cite[App.~2]{Z2} it is shown that $\alpha:F\to E$ is special if and only if for all $a,b\in F$
\begin{enumerate}[{\rm (1)}]
\item $a (E) b \implies \alpha(a)(W,SW)\alpha(b)$,
\item $a (S) b \implies \alpha(a)(SW,S)\alpha(b)$,
\item $a (NE) b \implies \alpha(a)(NE,N,NW,W,SW)\alpha(b)$,
\item $a (SE) b \implies \alpha(a)(SW)\alpha(b)$.
\end{enumerate}
Here the letter combinations E, S, SW etc. in the brackets refer to the usual wind directions and they are mutually exclusive.
For example, $a(W)b$ means that $a$ occurs strictly before $b$ in the same row and $a(SW)b$  means that $a$ occurs in a row
strictly below $b$ and in a column strictly to the left of $b$. Furthermore, ``$a(A,B)b$" means ``$a(A)b$ or $a(B)b$" and
similar for more than two wind directions. In \cite{Z2} it is also pointed out that property (4) actually follows from (1) and (2).
Although we will not use this equivalent characterisation, it can be useful to get an idea of what it means for a diagram mapping
to be special. If $\alpha$ is special, then $S_E\circ\alpha$ is semi-standard and $\alpha$ is admissible.
The converse is not true as can be seen by taking $\alpha$ the identity map from a row diagram with more than one box to itself.

\begin{thmgl}[{\cite[Thm~2.4]{Donin1}, \cite[Thm~1]{Donin2}}]\label{thm.homspace_basis}\ 
\begin{enumerate}[{\rm(i)}]
\item The elements $e^*[T]$ with $T$ in a set of (admissible) representatives of the special semi-standard tableaux of
shape $F$ and weight $\mu$ form a basis of $e^*Af$.
\item For every special semi-standard tableau $S$ of shape $F$ and weight $\mu$, there is precisely one
special diagram mapping $\alpha:F\to E$ such that $S=S_E\circ\alpha$ and all special diagram mappings occur in this way.
\end{enumerate}
\end{thmgl}
\begin{proof}
Assume $\alpha:F\to E$ is special. Then it follows that \hbox{$S=S_E\circ\alpha$} is ordered, since the ordering $\le$ is linear
on the rows and columns of $F$. Furthermore, $\alpha^{-1}:E\to F$ is also special. From this it follows that
if $b$ is strictly below $a$ in the same column of $F$, then $\alpha(b)$ occurs in a row strictly below $\alpha(a)$, i.e.
$S$ is semi-standard. Since $\alpha^{-1}$ has the analogous property, $\alpha$ has property (b), i.e. $S$ is special.
The image of the $i$-th row of $E$ under $\alpha^{-1}$ is $S^{-1}(i)$, and, since the ordering $\le$ is linear
on the rows of $E$, $\alpha^{-1}$ is completely determined by the images of the rows of $E$ under $\alpha^{-1}$.
So for every special semi-standard tableau $S$ of shape $F$ and weight $\mu$, there is
at most one special diagram mapping $\alpha:F\to E$ such that $S=S_E\circ\alpha$.
By \cite[Thm~1]{Z1} the number of special diagram mappings is equal to $\dim\Hom_A(Ae,Af)$ which is equal
to $\dim e^*Af$ by Lemma~\ref{lem.homspaces}.
So to prove (i) and (ii) it suffices to show that the elements given in (i) are linearly independent.

Recall that our representatives $T$ are supposed to be admissible, that is $\alpha_T$ must satisfies property (b') from Lemma~\ref{lem.special}.
Let $C_{T_E}\le G$ and $C_F\le\Sym(F)$ be the column stabilisers of $T_E$ and $F$ and let $T$ be as above. Then we have
$$e^*[T]=\sum_{g\in C_{T_E},\, \sigma\in C_F}{\rm sgn}(g){\rm sgn}(\sigma)e_2g\{T\sigma\}=\sum_{\pi\in\tilde C_F,\, \sigma\in C_F}{\rm sgn}(\pi){\rm sgn}(\sigma)e_2\{T\pi\sigma\}\,,$$
where $\tilde C_F=T^{-1}C_{T_E}T=\alpha_T^{-1}C_E\alpha_T\le\Sym(F)$ is the stabiliser of the sets $\alpha_T^{-1}(E^i)$, $E^i$ the $i$-th column of $E$.
If, for $\pi\in\tilde C_F$, $S_E\circ\alpha_{T\pi}$ has a repeated entry in some column, then $\sum_{\sigma\in C_F}{\rm sgn}(\sigma)e_2\{T\pi\sigma\}=0$.
By Lemma~\ref{lem.independent} it suffices to show that $e_2\{T\}$ occurs with strictly positive coefficient in $e^*[T]$ and $e_2\{T\}\le e_2\{T\pi\sigma\}$
for all $\pi\in\tilde C_F$ such that $S_E\circ\alpha_{T\pi}$ has no repeated entry in any column, and all $\sigma\in C_F$.

For $\pi\in\tilde C_F$ with this property let $\sigma_\pi\in C_F$ be the element such that \hbox{$S_E\circ\alpha_{T\pi\sigma_\pi}$} is (strictly) column ordered.
Then $S_E\circ\alpha_{T\pi\sigma_\pi}<S_E\circ\alpha_{T\pi\sigma}$ for all $\sigma\in C_F\sm\{\sigma_\pi\}$ by Lemma~\ref{lem.order}. So it suffices to show that
$e_2\{T\}$ occurs with strictly positive coefficient in $e^*[T]$ and that for $\pi$ as above $e_2\{T\}\le e_2\{T\pi\sigma_\pi\}$.
Let $\pi\in\tilde C_F$ such that $S_E\circ\alpha_{T\pi}$ has no repeated entry in any column.
If $\pi\in C_F$, then $\sigma_\pi=\pi^{-1}$ and, ${\rm sgn}(\pi){\rm sgn}(\sigma_\pi)e_2\{T\pi\sigma_\pi\}=e_2\{T\}$. Now assume $\pi\notin C_F$.

We will finish by showing that $e_2\{T\}<e_2\{T\pi\sigma_\pi\}$. Let $a_1=(i_1,j_1)$ be the first cell of $F$ in the order of the standard scan which
is moved to another column by $\pi^{-1}$. So $a_1$ is the first cell whose value $r=S_E(\alpha_T(a_1))$ has moved to another column in
$S_E\circ\alpha_T\pi$. First we prove the following claim.\\
{\bf Claim}.\ {\it If $a=(i,j)$ and $\pi(a)$ are not in the same column, then we have $S_E(\alpha_T(\pi(a)))\ge S_E(\alpha_T(i_1,j))$.}
\vspace{-2mm}
\begin{proof}
Assume $a$ has the stated property. From the definition of $a_1$ it follows that $\pi(a)$ has row index $\ge i_1$.
If $\pi(a)$ has column index $>j$, then the semi-standardness of $S_E\circ\alpha_T$ gives us the result.
So we assume now that $\pi(a)$ has column index $<j$. Put $D=\alpha_T^{-1}(D')$, where $D'$ is the column of $E$ to which $\alpha_T(a)$ belongs.
Note that since $\alpha_T$ has properties (a) and (b'), the inverse images of the columns of $E$ under $\alpha_T$ are vertical strips (see \cite{Mac}).
Furthermore, they are stable under $\pi$. Note also that $S_E(b)$ is the row index of $b$ in $E$, so a cell of $D$ in a lower row than another cell of
$D$ must contain a strictly bigger number. Since the intersection of $D$ with the $j$-th
column of $F$ is not stable under $\pi$, it is also not stable under $\pi^{-1}$. So for some $b\in D$ in the $j$-th column of $F$, $\pi^{-1}(b)$
is not in the $j$-th column. By the definition of $a_1$, $b$ has row index $\ge i_1$. So $S_E(\alpha_T(i_1,j))\le S_E(\alpha_T(b))$, by the
semi-standardness of $S_E\circ\alpha_T$. Now $\pi(a)$ occurs in a row strictly below $b$, since its column index is $<j$ and $D$ is a vertical strip.
So $S_E(\alpha_T(b))<S_E(\alpha_T(\pi(a)))$. 
\end{proof}\vspace{-2mm}
\noindent From the claim and the choice of $a_1$ it immediately follows that $S_E\circ\alpha_T$ and $S_E\circ\alpha_T\pi\sigma_\pi$ have the same
first $i_1-1$ rows, and $$S_E(\alpha_T(\pi\sigma_\pi(i_1,j)))\ge S_E(\alpha_T(i_1,j))\text{\ for all $j$, with equality if\ }j<j_1.\eqno(*)$$

Now let $j_0,\ldots,j_2$ be the positions in the $i_1$-th row where $S_E\circ\alpha_T$ has an $r$. By (*) these are the only positions in the $i_1$-th row
where $S_E\circ\alpha_T\pi\sigma_\pi$ could have an $r$. Note that $j_0\le j_1\le j_2$. Now let $a$ be any cell of $S_E\circ\alpha_T$ which contains an $r$
such that $\pi^{-1}(a)$ has column index in $\{j_0,\ldots,j_2\}$. If the column index of $a$ is $>j_2$, then, by the semi-standardness of $S_E\circ\alpha_T$,
its row index is $<i_1$. So, by the definition of $a_1$, $\pi^{-1}(a)$ is in the same column as $a$ which is impossible. Now assume $\pi^{-1}(a)$ occurs in a
column strictly to the right of $a$. Put $D=\alpha_T^{-1}(D')$, where $D'$ is the column of $E$ to which $\alpha_T(a)$ belongs. Since $D$ is a vertical strip
$\pi^{-1}(a)$ has row-index strictly less than that of $a$ and must contain a number $<r$. So, by the semi-standardness of $S_E\circ\alpha_T$, its row index
is $<i_1$. By the definition of $a_1$, $\pi^{-1}(\pi^{-1}(a))$ is in the same column as $\pi^{-1}(a)$. If its row index would be $\ge i_1$, then
$D$ would have to contain another cell than $a$ with an $r$, since it is a vertical strip. This is impossible, so $\pi^{-1}(\pi^{-1}(a))$
has row index $<i_1$. But then we could keep applying $\pi^{-1}$ and stay in the same column. This contradicts the fact that $\pi^{-1}$
has finite order. So if $\pi^{-1}(a)$ has column index in $\{j_0,\ldots,j_2\}$, then the same is true for $a$.
Furthermore, if this were true for $a_1$, then $\pi^{-1}(a_1)$ would have to occur in a column strictly to the left of $a_1$.
Then it follows from the definition of $a_1$ that $S_E\circ\alpha_T\pi$ would have two $r$'s in the column containing $\pi^{-1}(a_1)$, 
contradicting our assumption on $\pi$.

So the number of occurrences of $r$ in the $i_1$-th row of $S_E\circ\alpha_T\pi\sigma_\pi$ is at least one less than in the $i_1$-th row of $S_E\circ\alpha_T$
and by (*) the number of occurrences of any $r'<r$ in the $i_1$-th row is the same. So we may finally conclude that $S_E\circ\alpha_T\pi\sigma_\pi>S_E\circ\alpha_T$.
\end{proof}

\begin{remsgl}
1. If we take $\ytableausetup
{mathmode, boxsize=1.1em}
F={\begin{ytableau}
	\ &\ \\
	\ &\ 
	\end{ytableau}}\ ,$ 
$\ytableausetup
{mathmode, boxsize=1.1em}
E={\begin{ytableau}
	\none&\ &\ \\
	\ &\ 
	\end{ytableau}}\ ,$\smallskip
	\ and $S$ the semistandard tableau of shape $F$ and weight $(2,2)$, then there is no admissible representative $4$-tableau for $S$ which is also standard.\\
2.\ Write $E=\nu/\tilde\nu$. Using Lemma~\ref{lem.special}, it is easy to see that a special tableau of shape $F$ and weight $\mu$ must satisfy the condition
from \cite[Cor~2]{Stem2} that $\tilde\nu+w(T_{\ge j})$ is dominant for all $j$. Since both sets count the same dimension, the two conditions are equivalent.\\
3.\ Donin considers tableaux of shape $E$ as diagram mappings $T:\Delta_t\to E$, where $\Sym_t$
acts via $\pi\cdot T=T\circ\pi^{-1}$ and he works with the modules $e^*A$ considered as left $\Sym_t$
modules via the inversion. In his approach one has to use the isomorphism $\Hom_A(e^*A,f^*A)\cong f^*Ae$,
and think of this space as having a spanning set labelled by diagram mappings $:E\to F$.
Furthermore, one then has to replace $(a,b,\alpha(a),\alpha(b))$ by $(\alpha(a),\alpha(b),a,b)$ in property (b) and (b') in Lemma~\ref{lem.special}.
\end{remsgl}

Of course the previous results are valid for any symmetric group $\Sym(X)$, $X$ a finite subset of $\mb N$ with $t$ elements.
Just redefine $T_F$ by writing by filling in the elements from $X$ in their natural order row by row from left to right and top to bottom
and replace ``$t$-tableau of shape $E$" by ``$X$-tableau of shape $E$": this is a tableau whose entries are the elements of $X$ (so its entries are distinct).

For $X\subseteq\{1,\ldots,t\}$ we consider $\Sym(X)$ as a subgroup of $\Sym_t$ by letting the permutations from $\Sym(X)$ fix everything outside $X$.
When we apply our previous results to $\Sym(X)$ we use $X$ as an extra subscript when necessary. The group algebra $A_X=k\Sym(X)$ is a subalgebra of $A$.
If $D$ is a skew tableau with $|X|$ boxes, then we denote the Young symmetriser associated to the standard tableau $T_{D,X}$ by $e_{D,X}$.

Let $\nu=(\nu_1,\ldots,\nu_m)$ be an $m$-tuple of integers $\ge0$ with sum $t$. For $i\in\{1,\ldots,m\}$, put $\Lambda_i=\{j+\sum_{h=1}^{i-1}\nu_h\,|\,1\le j\le\nu_i\}$.
Then the {\it Young subgroup} $\Sym_\nu$ of $\Sym_t$ associated to $\nu$ is the simultaneous stabiliser of the sets $\Lambda_1,\ldots,\Lambda_m$.
So $\Sym_\nu\cong\prod_{i=1}^m\Sym_{\nu_i}$.
Let $\lambda\supseteq\mu$ be partitions with $E=\lambda/\mu$. Then there is a 1-1 correspondence between ordered tableaux of shape $E$ with entries
$\le m$ and sequences of partitions $\lambda^0,\ldots,\lambda^m$ with $\mu=\lambda^0\subseteq\lambda^1\subseteq\cdots\subseteq\lambda^m=\lambda$.
Indeed if $P$ is such a tableau, then $(\mu\cup P^{-1}(\{1,\ldots,i\}))_{1\le i\le m}$ is
such a sequence of partitions. Conversely we can construct $P$ from such a sequence: just fill the boxes of $\lambda^i/\lambda^{i-1}$ with $i$'s
for all $i\in\{1,\ldots,m\}$. So we can express the well-known rule for restricting skew Specht modules to Young subgroups in terms of tableaux $P$ as above.
We say that a $t$-tableau $T$ of shape $E$ {\it belongs to $P$} if $T^{-1}(\Lambda_i)=P^{-1}(i)$ for all $i\in\{1,\ldots,m\}$. Then $T$ will be standard if and only if the $T|_{P^{-1}(i)}$ are standard.
Every standard tableau of shape $E$ belongs to some ordered tableau of shape $E$ and weight $\nu$. If $P$ is an ordered tableaux of shape $E$ and weight $\nu$,
then we define $T_P$ to be the tableau of shape $E$ with $T_P|_{P^{-1}(i)}=T_{P^{-1}(i),\Lambda_i}$.
Note that $T_P$ is a standard tableau which belongs to $P$.

Let $P$ and $Q$ be ordered tableaux of shapes $E$ and $F$, both of weight $\nu\in\mb Z^m$. Then a diagram mapping $\alpha:F\to E$
with $P\circ\alpha=Q$ determines an $m$-tuple of tableaux $(S_{P^{-1}(1)}\circ\alpha_1,\ldots,S_{P^{-1}(m)}\circ\alpha_m)$ (*),
where $\alpha_i:Q^{-1}(i)\to P^{-1}(i)$ is the restriction of $\alpha$ to $Q^{-1}(i)$.
We will say that $\alpha$ {\it represents} (*).
Notice that all the $m$-tuples (*) have the same tuple of shapes and the same tuple of weights. We express this
by saying that the tuple of tableaux has {\it shapes determined by $Q$ and weights determined by $P$}.
Similarly, if $T$ is a $t$-tableau of shape $F$ which belongs to $Q$,
then we say that $T$ {\it represents} (*), where $\alpha_i=T_{P^{-1}(i),\Lambda_i}^{-1}\circ T|_{Q^{-1}(i)}$.
So if we cut $T$ to pieces according to $Q$, then $\alpha_i:Q^{-1}(i)\to P^{-1}(i)$ above is just the diagram mapping corresponding to the $i$-th piece.
Note that the ``union" of the above $\alpha_i$ is $T_P^{-1}\circ T$. When the tableaux $S_{P^{-1}(i)}\circ\alpha_i$ are special semi-standard,
we require the $\alpha_i$ (or $T_i=T|_{Q^{-1}(i)}$) to be admissible.

Let $\nu$ be as above. If $H$ is a group and $U$ an $H$-module, then $U_H$, sometimes called the space of ``coinvariants",
is defined as the largest quotient of $U$ which has trivial $H$-action, i.e. the quotient of $U$ by the subspace spanned by the elements $gx-x$, $x\in U, g\in\ H$.

\begin{propgl}\label{prop.coinvariants}
Assume that $E$ and $F$ are ordinary Young tableaux. Let $\nu$ and $\Sym_\nu$ be as above. Then the canonical images of the elements $[T_P]\ot[T]$,
where for each pair $(P,Q)$ with $P$ and $Q$ ordered tableaux of shapes
$E$ and $F$, both of weight $\nu$, $T$ goes through a set of representatives for the $m$-tuples of special semi-standard tableaux
with shapes determined by $Q$ and weights determined by $P$, form a basis for $(Ae\ot Af)_{\Sym_\nu}$.
\end{propgl}
\begin{proof}
Let $\Omega_E$ be the set of ordered tableaux of shape $E$ and weight $\nu$. 
For $P\in\Omega_E$ put $M_P=\ot_{i=1}^mA_{\Lambda_i}e_{P^{-1}(i),\Lambda_i}$
and let $\theta_P:M_P\to Ae$ be the linear map which sends $\ot_{i=1}^m[T_i]$, $T_i$ standard of shape $P^{-1}(i)$ with entries in $\Lambda_i$,
to $[T]$ where $T$ is the (standard) tableau obtained by piecing the tableaux $T_i$ together according to $P$.
Then it follows from the basis theorem for $Ae$ that $Ae=\bigoplus_{P\in\Omega_E}\theta_P(M_P)$.
By \cite[Thm.~3.1]{JP} and a straightforward induction argument there is a total ordering $P_1<P_2<\cdots<P_p$ of $\Omega_E$ such that with
$N_j=\bigoplus_{h=1}^j\theta_{P_h}(M_{P_h})$ we have that for all $j\in\{1,\ldots,p\}$ $N_j$ is a $\Sym_\nu$-submodule and the natural map
$\ov\theta_{P_j}:M_{P_j}\to N_j/N_{j-1}$ is an isomorphism of $\Sym_\nu$-modules.
In particular, if $T$ is a $t$-tableau which belongs to $P_j$, then $[T]\in N_j$  
and the canonical image of $[T]$ in $N_j/N_{j-1}$ is the image of $\ot_{i=1}^m[T|_{P^{-1}(i)}]$ under $\ov\theta_{P_j}$.
Similar remarks apply to analogously defined $\Omega_F$ and,
for $Q\in\Omega_F$, $M_Q$ and $\theta_Q$. So (redefining the $P_j$) there is a total ordering $(P_1,Q_1)<(P_2,Q_2)<\cdots<(P_{pq},Q_{pq})$ of
$\Omega_E\times\Omega_F$ such that with (redefining) $N_j=\bigoplus_{h=1}^j\theta_{P_h}(M_{P_h})\ot\theta_{Q_h}(M_{Q_h})$ we have that for each
$j\in\{1,\ldots,pq\}$ $N_j$ is a $\Sym_\nu$-submodule and the natural map $\ov\theta_{P_j}\ot\ov\theta_{Q_j}:M_{P_j}\ot M_{Q_j}\to N_j/N_{j-1}$ is
an isomorphism of $\Sym_\nu$-modules.

Denote for each $P\in\Omega_E$ and $Q\in\Omega_F$ the given set of representative $t$-tableaux by $\Gamma_{PQ}$.
Let $\pi_j:N_j\to N_j/N_{j-1}$ be the natural map. By Theorem~\ref{thm.homspace_basis}, Lemma~\ref{lem.homspaces}(i)
and the fact that $\ov\theta_{P_j}\ot\ov\theta_{Q_j}$ is a homomorphism of $\Sym_\nu$-modules, the canonical images of the elements
$\pi_j([T_P]\ot[T])$, $T\in\Gamma_{P_jQ_j}$, in $(N_j/N_{j-1})_{\Sym_\nu}$ form a basis for $(N_j/N_{j-1})_{\Sym_\nu}$. When applying
Lemma~\ref{lem.homspaces}(i) we omitted the sum over $\Sym(\Lambda_i)$ coming from the definition of $\psi$ after moving
$e_{Q_j^{-1}(i),\Lambda_i}^*$ to the left as $e_{Q_j^{-1}(i),\Lambda_i}$, since we work with coinvariants rather than invariants.
Now the assertion follows by a straightforward induction.
\end{proof}

\begin{remgl}\label{rem.coinvariants}
The result \cite[Thm.~3.1]{Donin1} which deals with restriction to Young subgroups is incorrect since it assumes that the $\theta_P(M_P)$ are
$\Sym_\nu$-submodules.
\end{remgl}

\subsection{Bases for the highest weight vectors}\label{ss.highest_weight_vecs}
We return to the notation of Section~\ref{s.charp}. In particular $m,r,s$ are fixed integers $\ge1$.
For $l\in\{1,\ldots,m\}$ we denote the matrix entry functions of the $l$-th matrix component on $\Mat_{rs}^m$ by $x(l)_{ij}$.
For $t$ an integer $\ge 0$ let $\Sigma_t$ be the set of $m$-tuples $\nu=(\nu_1,\ldots,\nu_m)$ of integers $\ge0$ with sum $t$.
Furthermore, if $\lambda$ is a partition, then we define $C_\lambda\le\Sym(\lambda)$ to be the column stabiliser of $\lambda$.

\begin{thmgl}\label{thm.highest_weight_vecs}
Let $\lambda,\mu$ be a partitions of $t$ with $l(\mu)\le r$ and $l(\lambda)\le s$.
For $\nu\in\Sigma_t$, $P,Q$ ordered tableaux of shapes $\lambda$ and $\mu$, both of weight $\nu$ and
$\alpha:\mu\to\lambda$ a diagram mapping such that $P\circ\alpha=Q$ define
$$u_{\nu,P,Q,\alpha}=\sum_{\pi\in C_\mu, \sigma\in C_\lambda}{\rm sgn}(\pi){\rm sgn}(\sigma)\prod_{a\in\mu}x(Q(a))_{r-\pi(a)_1+1,\,\sigma(\alpha(a))_1}\,,$$
where for $b\in\mu$, $b_1$ denotes the row index of $b$ in $\mu$ and similar for $b\in\lambda$.
Then the elements $u_{\nu,P,Q,\alpha}$, where for each $P,Q,\nu$ as above $\alpha$ goes through a set of representatives 
for the $m$-tuples of special semi-standard tableaux with shapes determined by $Q$ and weights determined by $P$,
form a basis of $k[\Mat_{rs}^m]^{U_r\times U_s}_{(-\mu^{\rm rev},\lambda)}$.
\end{thmgl}

\begin{proof}
Let $V=k^r$ and $W=k^s$ be the natural modules of $\GL_r$ and $\GL_s$. Then $\Mat_{rs}=V\ot W^*$ and $\Mat_{rs}^*=V^*\ot W$.
So $k[\Mat_{rs}^m]=\bigoplus_{t\ge0}S^t\big((V^*\ot W)^m\big)=\bigoplus_{t\ge0,\nu\in\Sigma_t}S^\nu(V^*\ot W)=\bigoplus_{t\ge0,\nu\in\Sigma_t}((V^*)^{\ot t}\ot W^{\ot t})_{\Sym_\nu}$,
where, for $U$ any vector space  $S^\nu(U)=\ot_{i=1}^mS^{\nu_i}(U)$.
Therefore $$k[\Mat_{rs}^m]^{U_r\times U_s}_{(-\mu^{\rm rev},\lambda)}=\bigoplus_{\nu\in\Sigma_t}\Big(((V^*)^{\ot t})^{U_r}_{-\mu^{\rm rev}}\ot (W^{\ot t})^{U_s}_\lambda\Big)_{\Sym_\nu}\,.$$
As is well-known, $((V^*)^{\ot t})_{-\mu^{\rm rev}}$ and $(W^{\ot t})_\lambda$ are the permutation modules associated to $\mu$ and $\lambda$, and
$((V^*)^{\ot t})^{U_r}_{-\mu^{\rm rev}}$ and $(W^{\ot t})^{U_s}_\lambda$ are the Specht modules $Ae_\mu$ and $Ae_\lambda$, where $A=k\Sym_t$.
To each $t$-tableau $T$ of shape $\mu$ we associate the highest weight vector $e_{1,T}v^*_T\in(V^*)^{\ot t}$, where $v^*_T$ is the basis tensor which
has $v_{r-i+1}^*$'s in the positions which occur as entries in the $i$-th row, and $e_{1,T}$ is the column anti-symmetriser associated to $T$.
We also associate to each $t$-tableau of $T$ shape $\lambda$ the highest weight vector $e_{1,T}w_T\in W^{\ot t}$, where $w_T$  is the basis tensor
which has $w_i$'s in the positions which occur as entries in the $i$-th row, and again $e_{1,T}$ is the column anti-symmetriser associated to $T$.
Then $[T]\mapsto e_{1,T}v^*_T:Ae_\mu\to((V^*)^{\ot t})^{U_r}_{-\mu^{\rm rev}}$
and $[T]\mapsto e_{1,T}w_T:Ae_\lambda\to(W^{\ot t})^{U_s}_\lambda$ are isomorphisms.
So by Proposition~\ref{prop.coinvariants} with $E=\lambda$ and $F=\mu$ the canonical images in $M=\Big(((V^*)^{\ot t})^{U_r}_{-\mu^{\rm rev}}\ot (W^{\ot t})^{U_s}_\lambda\Big)_{\Sym_\nu}$
of the elements
$$e_{1,T}v^*_T\ot e_{1,T_P}w_{T_P}=\sum_{\pi\in C_\mu, \sigma\in C_\lambda}{\rm sgn}(\pi){\rm sgn}(\sigma)v^*_{T\pi^{-1}}\ot w_{T_P\sigma^{-1}}\,,$$
where for each $P,Q,\nu$ as above $T$ goes through a set of representatives 
for the $m$-tuples of special semi-standard tableaux with shapes determined by $Q$ and weights determined by $P$, form a basis of $M$.
Here we put in the inverses for convenience below. Now we change from representative tableaux $T$ to representative diagram mappings $\alpha$
via $\alpha=T_P^{-1}\circ T$ and take basis elements of $V$ and $W$ which occur in the same tensor position together: $v^*_{T\pi^{-1}}$ has
$v^*_{r-\pi(a)_1+1}$ in position $T(a)$ and $w_{T_P\sigma^{-1}}$ has $w_{\sigma(b)_1}$ in position $T_P(b)$, and those positions are the same
if and only if $b=\alpha(a)$. Finally, $T(a)\in \Lambda_{Q(a)}$, since $T$  belongs to $Q$. So $v^*_{r-\pi(a)_1+1}\ot w_{\sigma(\alpha(a))_1}$
becomes $x(Q(a))_{r-\pi(a)_1+1,\,\sigma(\alpha(a))_1}$.
\end{proof}
The next corollary gives a much simpler (but bigger) spanning set for the space of highest vectors $k[\Mat_{rs}^m]^{U_r\times U_s}_{(-\mu^{\rm rev},\lambda)}$. Of course it can, like the above theorem, be combined with Theorem~\ref{thm.surjective_pullback} and Lemma~\ref{lem.reduction_to_nilpotent_cone} to give spanning sets for the vector space $k[\mc N_n]^{U_n}_{[\lambda,\mu]}$ and the $k[\gl_n]^{\GL_n}$-module $k[\gl_n]^{U_n}_{[\lambda,\mu]}$.
\begin{corgl}
Let $\alpha=1^{\mu_1}2^{\mu_2}\cdots$ and $\beta=1^{\lambda_1}2^{\lambda_2}\cdots$ be the standard scans of $S_\mu$ and $S_\lambda$. Then the elements
$$\sum_{\pi\in C_{T_\mu},\sigma\in C_{T_\lambda}}{\rm sgn}(\pi){\rm sgn}(\sigma)\prod_{i=1}^tx(\gamma_i)_{r-\alpha_{\pi(i)}+1,\beta_{\sigma(\tau(i))}}\,,$$
where $\gamma\in\{1,\ldots,m\}^t$ and $\tau\in\Sym_t$ form a spanning set of $k[\Mat_{rs}^m]^{U_r\times U_s}_{(-\mu^{\rm rev},\lambda)}$.
\end{corgl}
\begin{proof}
In the proof of Theorem~\ref{thm.highest_weight_vecs} we take for each $\nu$ the bigger spanning set
$$e_{1,S}v^*_S\ot e_{1,T}w_{T}=\sum_{\pi\in C_S, \sigma\in C_T}{\rm sgn}(\pi){\rm sgn}(\sigma)v^*_{\pi^{-1}S}\ot w_{\sigma^{-1}T}\,,$$
where $S$ and $T$ are any $t$-tableaux of shape $\mu$ and $\lambda$.
Write $T=\rho^{-1}T_\mu$, $S=\tau^{-1}T_\lambda$ for $\rho,\tau\in\Sym_t$.
Then we get $$e_{1,S}v^*_S\ot e_{1,T}w_T=\sum_{\pi\in C_{T_\mu}, \sigma\in C_{T_\lambda}}{\rm sgn}(\pi){\rm sgn}(\sigma)\ot_{i=1}^tv^*_{r-\alpha_{\pi(\rho(i))}+1}\ot\ot_{i=1}^t w_{\beta_{\sigma(\tau(i))}}\,,$$
which corresponds to the element
$$\sum_{\pi\in C_{T_\mu}, \sigma\in C_{T_\lambda}}{\rm sgn}(\pi){\rm sgn}(\sigma)\prod_{i=1}^tx(\gamma_i)_{r-\alpha_{\pi(\rho(i))}+1,\beta_{\sigma(\tau(i))}}\in k[\Mat_{rs}^m]\,,$$
where $\gamma\in\{1,\ldots,m\}^t$ is the tuple with $i\in \Lambda_{\gamma_i}$ for all $i$. Recall that the $\Lambda_i$ depend on $\nu$ and note that $\gamma$ determines $\nu$. Now we observe that if we allow arbitrary tuples $\gamma\in\{1,\ldots,m\}^t$ we can take $\rho=\id$. So we obtain the assertion. 
\end{proof}

Recall te definition of the map $\varphi_{r,s,n,m}$ from Section~\ref{s.charp}.

\begin{corgl}
Let $\chi=[\lambda,\mu]$ be a dominant weight in the root lattice, $l(\mu)\le r$, $l(\lambda)\le s$, $|\lambda|=|\mu|=t$, $r+s\le n$.
Then the pull-backs of the elements $u_{\nu,P,Q,\alpha}$, $\nu,P,Q,\alpha$ as in Theorem~\ref{thm.highest_weight_vecs}, along $\varphi_{r,s,n,m}:\mc N_{n,m}\to\Mat_{rs}^m$
span the vector space $k[\mc N_{n,m}]^{U_n}_\chi$.
\end{corgl}
\begin{proof}
This follows immediately from Theorem~\ref{thm.highest_weight_vecs} and Theorem~\ref{thm.surjective_pullback}.
\end{proof}

\begin{corgl}
Let $\chi=[\lambda,\mu]$ be a dominant weight in the root lattice, $l(\mu)\le r$, $l(\lambda)\le s$, $|\lambda|=|\mu|=t$, $r+s\le n$.
Then the pull-backs of the elements $u_{\nu,P,Q,\alpha}$, $\nu,P,Q,\alpha$ as in Theorem~\ref{thm.highest_weight_vecs}, along $\varphi_{r,s,n,n-1}:\Mat_n\to\Mat_{rs}^{n-1}$
span the $k[\Mat_n]^{\GL_n}$-module $k[\Mat_n]^{U_n}_\chi$.
\end{corgl}
\begin{proof}
This follows from the previous corollary with $m=n-1$, Lemma~\ref{lem.reduction_to_nilpotent_cone} and Remark~\ref{rems.surjective_pullback}.1.
\end{proof}

\begin{remsgl}\label{rems.highest_weight_vecs}
1.\ It is instructive to consider some special cases. For example, in the case $t=m$ and $\nu$ the all-one vector, the highest weight vectors of multidegree $\nu$
are labelled by pairs $(P,Q)$ of standard tableaux of shape $\lambda$ and $\mu$.
Another example is the case that $\lambda$ consists of one row or column. Then there is for each $\nu$
only one $P$ and for each $Q$ there is at most one tuple of special semi-standard tableaux with shapes determined by $Q$ and weights determined by $P$.
The $Q$ which have such a tuple are the semi-standard tableaux of weight $\nu$ if $\lambda$ is a row and the row semi-standard tableaux of weight $\nu$ if $\lambda$ is a column. 
Similar remarks apply to the case that $\mu$ consists of one row or column. The last two cases extend to prime characteristic, see Theorem~\ref{thm.basis_special_weights}
and Remark~\ref{rems.basis_special_weights}.1.\\
2.\ 
Corollary~3 to Theorem~\ref{thm.highest_weight_vecs} was already stated by Donin in \cite[after Thm~3]{Donin2}, \cite[Prop.~4.1]{Donin1}. He worked with $S(\Mat_n)$ rather than $k[\Mat_n]$,
so our $x_{ij}\in\Mat_n^*$ corresponds to his $e_{ji}\in\Mat_n$.
Note that pulling the $u_{\nu,P,Q,\alpha}$ back just amounts to interpreting $x(Q(a))_{ij}$ as the $(i,j)$-th entry of the $Q(a)$-th matrix power and replacing $r-\pi(a)+1$ by $n-\pi(a)+1$. In particular, these pulled-back functions don't depend on the choice of $r$ and $s$. In case of $\mc N_{n,m}$ the spanning sets are bases in all degrees for $n\ge (m+1)\max(r,s)$.
In case of $\mc N_n$ one can only say that in a fixed degree $d$ the spanning sets will be bases if $n\ge l(\lambda)+l(\mu)+d-t$, where $t=|\lambda|=|\mu|$. This follows from Remark~\ref{rems.surjective_pullback}.2.

Donin claimed in \cite[Prop.~4.1]{Donin1} and \cite[Thm.~3]{Donin2} that the spanning sets obtained above are always bases, but this is easily seen to be incorrect.
For example, for $\gl_4$ and $\lambda=\mu=1^2$ we deduce, using the Hesselink-Peterson formula \cite{Hes} or the Lascoux-Sch\"utzenberger-charge \cite{LS} on tableaux, that the degree $3$ piece of $k[\mc N_4]^{U_4}_{[\lambda,\mu]}$ is $0$, but our spanning set contains one element of degree $3$.
In case $\lambda$ or $\mu$ is a row the spanning set is a basis, see Remark~\ref{rems.basis_special_weights}.1.

Finding explicit homogeneous bases for all the spaces $k[\mc N_n]^{U_n}_{\mc \chi}$ (or more generally $k[\ov{\mc O}_\eta]^{U_n}_{\mc \chi}$) is still an open problem. If one tries to find them as subsets of the above spanning sets this is combinatorially already a challenging problem.
In the case of the $\GL_n$-modules $V^{\ot r}\ot(V^*)^{\ot s}$, $V=k^n$, there is a similar problem of finding bases for the vector spaces $(V^{\ot r}\ot(V^*)^{\ot s})^{U_n}_{[\lambda,\mu]}$. In \cite{BCHLLS} this was done for $n\ge l(\lambda)+l(\mu)+r-|\lambda|$. 
In this case there is at least a good candidate indexing set for arbitrary $n$: the up-down staircase tableaux of \cite{Stem1}.\\
3. Note that in Theorem~\ref{thm.highest_weight_vecs} we can choose each $\alpha$ the unique representative such that for all $i$ $\alpha_i$ is special, i.e. a ``picture" in the sense of \cite{JP} and \cite{Z1}.\\
4. Corollary~1 to Theorem~\ref{thm.highest_weight_vecs} proves a weaker version of the ``conjecture" in \cite[Sect.~4]{T1}: in the notation there, with $\chi=[\lambda,\mu]$, the elements $$\vartheta\big(\psi_t((\tau,\id)\cdot E_\chi)\cdot s_{i_1}\ot\cdots\ot s_{i_t}\big)\,,$$ $2\le i_1,\ldots,i_t\le n$, $\tau\in\Sym_t$ generate the $k[\gl_n]^{\GL_n}$-module $k[\gl_n]^{U_n}_\chi$.
This follows by pulling the spanning set of the corollary back to the nilpotent cone taking $m=n-1$, using the fact that $(X^l)_{ij}=\pm(\partial_{ji}s_{l+1})(X)$ for all $X\in\mc N_n$, see \cite[Cor to Thm~1]{T2}, and applying Lemma~\ref{lem.reduction_to_nilpotent_cone}. Of course one can also use the $\{\id\}\times\Sym_t$-conjugates of $E_\chi$. Just take $\tau=\id$ and $\rho$ in the proof of the corollary. The original conjecture is false, see \cite[Rem~2.5]{T2}.
\end{remsgl}

\subsection{Several matrices}\label{ss.several_matrices}
In this final section we look at highest weight vectors in the coordinate ring of the space of several matrices $\Mat_n^l$ under the diagonal conjugation action of $\GL_n$. In order to be able to apply the graded Nakayama Lemma we need to work with the ``null-scheme" rather than the null-cone. We will denote an $l$-tuple of $n\times n$-matrices $(X_1,\ldots,X_l)$ by $\un X$.

We recall some results from \cite[Sect.~4]{Br}. For $i$ an integer $\ge0$ let $\mc X_i$ be the set of sequences of length $\le i$ with entries in $\{1,\ldots,l\}$ and let $\mc X_i'$ be $\mc X_i$ with the empty sequence omitted. For $\eta\in\mc X_i$ of length $j\le i$ define $f_\eta:\Mat_n^l\to\Mat_n$ by $f_\eta(\un X)=X_{\eta_1}\cdots X_{\eta_j}$.
By the Razmyslov-Procesi Theorem the algebra $k[\Mat_{rn}\times\Mat_{ns}\times\Mat_n^l]^{\GL_n}$ is generated by the functions
$(A,B,\un X)\mapsto\tr(f_\eta(\un X))$ and $(A,B,\un X)\mapsto (Af_\xi(\un X)B)_{ij}$, $\eta\in\mc X_{n^2}'$, $\xi\in\mc X_{n^2-1}$, $i\in\{1,\ldots,r\}$ and $j\in\{1,\ldots,s\}$.
Now let $\mc M_n$ be the closed subscheme of $\Mat_n^l$ corresponding to the ideal of $k[\Mat_n^l]$ generated by the functions $\un X\mapsto\tr(f_\eta(\un X))$. Then it follows from the above that for $m=|\mc X_{n^2-1}'|$ the restriction of the morphism
$$\psi_{r,s,n,l}:(A,B,\un X)\mapsto(Af_\xi(\un X)B)_{\xi\in \mc X_{n^2-1}'}
:Y_{r,s,n}\times\Mat_n^l\to\Mat_{rs}^m$$
to $Y_{r,s,n}\times\mc M_n$ is a $\GL_n$-quotient morphism onto its scheme-theoretic image $\mc W_{n,l}$.
Note that we omitted the empty sequence from $\mc X_{n^2-1}$, since we passed to $Y_{r,s,n}$, the variety of pairs of matrices $(A,B)\in\Mat_{rn}\times\Mat_{ns}$ with $AB=0$.

Analogous to the case of one matrix we will identify $\Mat_n^l$ with the closed subvariety $\{(E_r,F_s)\}\times\Mat_n^l$ of $Y_{r,s,n}\times\Mat_n^l$ and denote the restriction of $\psi_{r,s,n,l}$ to $\Mat_n^l$ again by $\psi_{r,s,n,l}$. Then the union of the $\GL_n$-conjugates of $\Mat_n^l=\{(E_r,F_s)\}\times\Mat_n^l$ is $\mc O\times\Mat_n^l$, where  $\mc O$ consists of the pairs $(A,B)\in Y_{r,s,n}$ with ${\rm rk}(A)=r$ and ${\rm rk}(B)=s$. The same holds with $\Mat_n^l$ replaced by $\mc M_n$. 
It follows that the comorphism of $\psi_{r,s,n,l}:\mc M_n\to\mc W_{n,l}$ is injective, since the natural map $k[Y_{r,s,n}\times\mc M_n]\to k[\mc O\times\mc M_n]$ is injective. Furthermore, the analogue of the identity for $\varphi_{r,s,n,m}$ at the beginning of the proof of Theorem~\ref{thm.surjective_pullback} holds for $\psi_{r,s,n,l}$.
Finally we apply the graded Nakayama Lemma to the $k[\Mat_n^l]^{\GL_n}$-module $k[\Mat_n^l]^{U_n}_\chi$ 
and we obtain

\begin{thmgl}
Let $\chi=[\lambda,\mu]$ be a dominant weight with coordinate sum zero and put $m=|\mc X_{n^2-1}'|$. Then the pull-back along $\psi_{r,s,n,l}:\Mat_n^l\to\Mat_{rs}^m$ of the spanning set of $k[\Mat_{rs}^m]^{U_r\times U_s}_{(-\mu^{\rm rev},\lambda)}$ from Theorem~\ref{thm.highest_weight_vecs} or the one from Corollary~1 is a spanning set of the $k[\Mat_n^l]^{\GL_n}$-module $k[\Mat_n^l]^{U_n}_\chi$.
\end{thmgl}
\begin{remsgl}
1.\ Of course $m$ above is huge, but if we are only interested in homogeneous highest weight vectors of degree $d$ say, then we can take $m=|\mc X_d'|$ above and combine the resulting elements with homogeneous elements of $k[\Mat_n^l]^{\GL_n}$ to obtain a spanning set for the vector space of homogeneous highest weight vectors of weight $\chi$ and degree $d$.\\
2.\ Much of Section~\ref{ss.several_matrices} generalises to prime characteristic, but it is not clear how to prove the analogue of Proposition~\ref{prop.good_pair} for several matrices.
\end{remsgl}

\noindent{\it Acknowledgement}. This research was funded by the EPSRC grant EP/L013037/1.

\bigskip

{\sc\noindent School of Mathematics,\\
University of Leeds, LS2 9JT, Leeds, UK.\\
{\it E-mail address : }{\tt R.H.Tange@leeds.ac.uk}
}

\end{document}